\documentclass[a4paper, reqno, 11pt]{amsart}

\usepackage[english]{babel}
\usepackage{amsmath}
\usepackage{mathrsfs}
\usepackage{amssymb}
\usepackage{amsthm}
\usepackage{enumerate}
\usepackage{ifthen}
\usepackage{bbm}
\usepackage{color}
\usepackage{xcolor}
\usepackage{graphicx}
\provideboolean{shownotes} 
\setboolean{shownotes}{true}
\usepackage{hyperref}
\usepackage{geometry}
\newcommand{\margnote}[1]{
\ifthenelse{\boolean{shownotes}}%
{\marginpar{\raggedright\tiny\texttt{#1}}}%
{}%
}

\newcommand{\hole}[1]{
\ifthenelse{\boolean{shownotes}}%
{\begin{center} \fbox{ \rule {.25cm}{0cm}
\rule[-.1cm]{0cm}{.4cm} \parbox{.85\textwidth}{\begin{center}
\texttt{#1}\end{center}} \rule {.25cm}{0cm}}\end{center}}
{}
}
\newtheorem{thm}{Theorem}[section]

\newtheorem{prop}[thm]{Proposition}
\newtheorem{lem}[thm]{Lemma}

\newtheorem{rem}[thm]{Remark}

\newtheorem{defn}[thm]{Definition}

\newcommand{\e}{\varepsilon}		       
\newcommand{\R}{\mathbb{R}}
\newcommand{\T}{\mathbb{T}^2}
\newcommand{\N}{\mathbb{N}}
\newcommand{\Z}{\mathbb{Z}}
\newcommand{\E}{\mathbb{E}}
\newcommand{\uv}{u^{\nu}}

\newcommand{\om}{\omega}
\newcommand{\omv}{\omega^{\nu}}
\newcommand{\dive}{\mathop{\mathrm {div}}}
\newcommand{\curl}{\mathop{\mathrm {curl}}}
\newcommand{\weakto}{\rightharpoonup}
\newcommand{\weaktos}{\stackrel{*}{\rightharpoonup}}
\newcommand{\de}{\,\mathrm{d}}

\newcommand{\omvep}{\omega^{\nu,\e}_{1}}
\newcommand{\omvei}{\omega^{\nu,\e}_{\infty}}
\newcommand{\omvepi}{\omega^{\nu,\e}_{0,1}}
\newcommand{\omveii}{\omega^{\nu,\e}_{0,\infty}}

\numberwithin{equation}{section}

\begin{document}

\title[Vanishing viscosity]{Strong convergence of the vorticity for the 2D Euler Equations in the inviscid limit}

\author[G. Ciampa]{Gennaro Ciampa}
\address[G. Ciampa]{Department Mathematik Und Informatik\\ Universit\"at Basel \\Spiegelgasse 1 \\CH-4051 Basel \\ Switzerland }
\email[]{\href{gennaro.ciampa@}{gennaro.ciampa@unibas.ch}}

\author[G. Crippa]{Gianluca Crippa}
\address[G. Crippa]{Department Mathematik Und Informatik\\ Universit\"at Basel \\Spiegelgasse 1 \\CH-4051 Basel \\ Switzerland}
\email[]{\href{gianluca.crippa@}{gianluca.crippa@unibas.ch}}

\author[S. Spirito]{Stefano Spirito}
\address[S. Spirito]{DISIM - Dipartimento di Ingegneria e Scienze dell'Informazione e Matematica\\ Universit\`a  degli Studi dell'Aquila \\Via Vetoio \\ 67100 L'Aquila \\ Italy}
\email[]{\href{stefano.spirito@}{stefano.spirito@univaq.it}}

\begin{abstract}
In this paper we prove the uniform-in-time $L^p$ convergence in the inviscid limit of a family $\omega^\nu$ of solutions of the $2D$ Navier-Stokes equations towards a renormalized/Lagrangian solution $\omega$ of the Euler equations.  We also prove that, in the class of solutions with bounded vorticity, it is possible to obtain a rate for the convergence of $\omv$ to $\om$ in $L^p$. Finally, we show that solutions of the Euler equations with $L^p$ vorticity, obtained in the vanishing viscosity limit, conserve the kinetic energy. The proofs are given by using both a (stochastic) Lagrangian approach and an Eulerian approach. 
\end{abstract}

\maketitle

\section{Introduction}
We consider the Cauchy problem for the two-dimensional incompressible Euler equations in vorticity formulation given by 
\begin{equation}\label{eq:vort}
\begin{cases}
\partial_t\omega+u\cdot\nabla\omega=0,\\
\omega|_{t=0}=\omega_0,
\end{cases}
\end{equation}
where $u$ is the velocity field and $\omega_0$ is a given initial datum. The velocity is recovered from the vorticity via the Biot-Savart law. A classical problem in fluid mechanics is the approximation in the limit $\nu\to 0$ of vanishing viscosity (also called inviscid limit) of solutions of \eqref{eq:vort} by solutions of the incompressible Navier-Stokes equations 
\begin{equation}\label{eq:vort_NS}
\begin{cases}
\partial_t \omega^\nu+u^\nu\cdot\nabla\omega^\nu=\nu\Delta \omega^\nu,\\
\omega^\nu|_{t=0}=\omega^{\nu}_0.
\end{cases}
\end{equation}

The goal of this paper is to study several problems related to the convergence of $\omega^{\nu}$ to $\omega$  when the equations \eqref{eq:vort} and \eqref{eq:vort_NS} are considered either on the two-dimensional torus or on the whole space.\\

Local-in-time existence of classical solutions of \eqref{eq:vort} with smooth initial data was proved by Lichtenstein in \cite{L}, while global-in-time existence was proved by Wolibner in \cite{W}. Assuming only integrability hypothesis on the initial vorticity, more precisely $\omega_0\in L^1\cap L^p$ for some $p> 1$, DiPerna and Majda in \cite{DPM} proved global existence of weak solutions. The results in \cite{DPM} were extended to the case of a finite Radon measure in $H^{-1}_{loc}$ with distinguished sign  in \cite{De} and to $\omega_0\in L^1$ in \cite{VW}. Uniqueness is known only for $p=\infty$ and was proved by Yudovi\v{c} in \cite{Y}. The uniqueness for unbounded vorticities is an old and outstanding open problem and only very recently some partial progress towards nonuniqueness has been achieved, see \cite{BM, BS, MeS, VI, VII}.
\\

Concerning the behaviour of the Navier-Stokes vorticity $\omv$ in the limit of vanishing viscosity, in the setting of DiPerna-Majda \cite{DPM} it holds that, up to a subsequence, there exists $\omega\in L^{\infty}(L^{p})$ such that 
\begin{equation}\label{eq:weaks}
\omv\weaktos \om\quad\mbox{ weakly* in } L^{\infty}(L^{p}).
\end{equation}
The limit $\omega$ is a distributional solution of \eqref{eq:vort} provided $p>4/3$. We are interested in the strong convergence of the vorticity, namely 
\begin{equation}\label{eq:strongintro}
\omv\to \om\quad\mbox{ strongly in } C(L^{p}),\,p\in[1,\infty).
\end{equation}
The upgrade of \eqref{eq:weaks} to \eqref{eq:strongintro} was proved by several authors in various settings. In particular, the case of a smooth initial datum is well-established, see e.g. \cite{C1} and \cite{M} and references therein. In less regular settings, we recall the result in \cite{CW1} for vortex-patch solutions and then for more general bounded solutions by requiring additional assumptions on the Euler path in \cite{CW2}. In the very recent paper \cite{CDE}, P. Constantin, T. Drivas and T. Elgindi proved the upgrade to strong convergence in the case of bounded vorticity without additional assumptions. Precisely, they proved that on the two-dimensional torus, if $\omega_0\in L^{\infty}$ and $\omega\in L^{\infty}(L^{\infty})$ is the unique bounded solution of \eqref{eq:vort}, then for any $1\leq q <\infty$
\begin{equation}\label{eq:intro1}
\omega^{\nu}\to\omega\quad\mbox{ strongly in }C(L^{q}). 
\end{equation}

In this paper we improve the result of \cite{CDE} by proving that both in the periodic setting and in the whole space setting, if $\omega_0\in L^{1}\cap L^{p}$ with $1\leq p<\infty$ and $\omega$ as in \eqref{eq:weaks} is a renormalized solution of \eqref{eq:vort} in the sense of DiPerna-Lions \cite{DPL}, then \eqref{eq:intro1} holds for for any $1\leq q \leq p$. We notice that the possibility of this improvement was already remarked in \cite[Remark 2]{CDE} and proved at the very same time of our paper and independently from us in \cite{SW} in the case of the torus and for $p>1$. \\

We give two proofs of the convergence result described above which are based on two different approaches: the Lagrangian approach and the Eulerian approach.\\
 
In the Lagrangian approach, we focus on the case of the two-dimensional flat torus and only consider $p>1$; contrary to \cite{SW}, we give a quantitative proof. Precisely, we prove that for any $\delta>0$, there exists $C=C(\delta,\omega_0)>0$ such that for $\nu$ small enough
\begin{equation}\label{eq:intro2}
\sup_{t\in(0,T)}\|\omv(t)-\om(t)\|_{L^{p}}\leq \delta+\frac{C(\delta,\omega_0)}{|\ln(\max\{\sqrt{\nu}, \|\uv-u\|_{L^{1}(L^1)}\})|}+\|\omega_0^{\nu}-\omega_0\|_{L^{p}}.
\end{equation}
We refer to Theorem \ref{teo:main2} for the rigorous statement. To obtain \eqref{eq:intro2} we first exploit the stochastic Lagrangian formulation of the incompressible Navier-Stokes equations (as in the paper by P. Constantin and G. Iyer \cite{CI}) and then we revisit the quantitative estimates for flows of Sobolev vector fields obtained by the second author and C. De Lellis in \cite{CDL} and their stochastic counterpart by N. Champagnat and P.-E. Jabin in \cite{CJ}, where a more general result on quantitative estimates for stochastic flows and their deterministic limit is given. In particular, the result is achieved by studying the zero-noise limit from stochastic towards deterministic flows of irregular vector fields. This result is of its own importance in the theory of stochastic flows. We refer to the recent monograph of C. Le Bris and P.-L. Lions \cite{LBL} for recent advances on stochastic flows of irregular vector fields.\\

Of course, \eqref{eq:intro2} is not fully quantitative since it depends implicitly on the difference of the velocities and some approximation of the initial datum. While the dependence on the approximation of the initial datum can be made quantitative by assuming addition regularity, e.g. we could assume $\omega_0\in H^{s}$ with $s>0$, the dependence on the the difference of the velocities is difficult to avoid, unless the initial datum $\omega_0\in L^{\infty}$. The  second main result of this note concerns the analysis of the rate of convergence when the initial vorticity is merely bounded. In \cite{CDE} it is proved that in the case of the two-dimensional torus, if $\omega_0\in L^{\infty}\cap B^{s}_{p,\infty}$, with $s>0$ and $p\geq 1$, ($B^{s}_{p,\infty}$ is the classical Besov space), then \begin{equation}\label{eq:cderateintro}
\sup_{0\leq t\leq T}\|\omega^{\nu}(t)-\omega(t)\|_{L^p}\leq C \nu^{\frac{s\exp{(-CT\|\omega^0\|_{L^{\infty}})}}{1+s\exp{(-CT\|\omega^0\|_{L^{\infty}})}}}.
\end{equation}
A crucial tool to obtain \eqref{eq:cderateintro} is the following {\em losing estimate}: if $\omega_0\in L^{\infty}\cap B^{s}_{p,\infty}$ then the solution $\omega^\nu(t)\in L^{\infty}\cap B^{s(t)}_{p,\infty}$ uniformly in $\nu$ with $s(t)=s\exp{(-Ct\|\omega_0\|_{L^{\infty}})}$. In \cite[Remark 2]{CDE} the authors notice that by the very same argument used to prove \eqref{eq:cderateintro} is possible to obtain for any $\omega_0\in L^{\infty}$ a rate of convergence. In the present paper, we also obtain a rate of convergence for any $\omega_0\in L^{\infty}$, but we use a different argument. Precisely, by using an Osgood-type argument as in the result of J.-Y. Chemin \cite{Ch}, arguing directly at Lagrangian level and using the continuity of translation in $L^1$ for $\omega_0$ we deduce that
there exist $\nu_0>0$ and a continuous function $\phi:\R^+\to\R^+$ with $\phi(0)=0$ such that for $\nu<\nu_0$
\begin{equation}\label{eq:intro3}
\sup_{t\in(0,T)}\|\omv(t)-\om(t)\|_{L^{p}(\T)}\lesssim \phi(\nu),
\end{equation}
where the implicit constant in the inequality \eqref{eq:intro3} grows with $T$ and $\phi$ is not in general explicit. We refer to Theorem \ref{teo:rate} for the rigorous statement. It is worth to point out that compared to \cite{CDE} we do not propagate any regularity of the initial data. Indeed, the rate of convergence is achieved again directly in the zero-noise limit of the stochastic flow. 
\bigskip

In the second part we use the Eulerian approach to prove strong convergence of the vorticity. In particular, we consider also the case $p=1$ and the case when the domain is the whole space.  The main theorem of this part is Theorem \ref{teo:strong_convergence_vort}. 
We note that the fact that renormalization implies strong convergence is already valid for the linear transport equation, see \cite[Theorem II.4]{DPL}. We extend to the Euler equations and to the case  $p=1$ the arguments in \cite{DPL}. Roughly speaking, the idea is to use the Radon-Riesz theorem combined with an Ascoli-Arzel\`a argument. Notice that our proof based on the Eulerian approach is not quantitative, as usual for compactness arguments. Recently, D. Bresch and P.-E. Jabin proved in \cite{BJ} quantitative compactness estimates for solutions of the continuity equation without using Lagrangian arguments and exploited them for the analysis of compressible fluids. We believe that extending these estimates to the context of the $2D$ Euler equations would be very interesting.\\

Finally, we comment on the extension from the flat torus to the whole space, which is crucial to address the fundamental question of the conservation of the energy: it allows to extend from the two-dimensional torus to the whole space the result of \cite{CFLS} on the conservation of kinetic energy for solutions of the Euler equations obtained as limit of vanishing viscosity when the initial vorticity is in $ L^{p}$. Indeed, as already noticed in \cite{CCS3}, the main issue in extending the result of \cite{CFLS} to the whole space is to obtain global strong convergence in $C(L^{2})$ of the velocity. Due to the lack of compact embedding this cannot be obtained by using the Aubin-Lions lemma, but it is obtained by exploiting a Serfati-type formula \cite{S}, which in turn requires the strong convergence of the vorticities. We refer to Theorem \ref{teo:energy}  for this result. 

\bigskip
\subsection*{Acknowledgments}
The authors gratefully acknowledge useful discussions with Ji\v r\'i \v Cern\'y, Peter Constantin, Theodore Drivas, Tarek Elgindi, Gautam Iyer, Pierre-Emmanuel Jabin, Christian Seis, Dimitrios Tsagkarogiannis, and Emil Wiedemann. 
This research has been supported by the ERC Starting Grant 676675 FLIRT.

\section{The Lagrangian approach}
The section is organized as follows: we first fix the notations and recall some of the notions needed, then we introduce the (stochastic) Lagrangian formulations of the Euler and the Navier-Stokes equations and finally we prove the two main theorems of this section, namely Theorem \ref{teo:main2} and Theorem \ref{teo:rate}. 
\subsection{Notations and Preliminaries}
We denote by $\T$ the flat torus, by $\mathsf{d}(\cdot,\cdot)$ the geodesic distance and by $\mathscr{L}^2$ the Haar measure on $\T$. We denote by $B_{r}(x)$ the geodesic ball centered at $x$ with radius $r$. We also identify the flat torus with the cube $[0,1)\times[0,1)$, in particular 
\begin{equation*}
\mathsf{d}(x,y):=\min\{|x-y-k|: k\in \Z^2\mbox{ such that }|k|\leq 2\}.
\end{equation*}
Notice that the Haar measure coincides with the Lebesgue measure on the square and functions on $\T$ can be identified with $1$-periodic functions on $\R^2$.\par
\subsection{The Lagrangian formulation of the Euler and the Navier-Stokes equations in two dimensions}
Let $T>0$ be finite but arbitrary and consider the $2D$ Euler equations in $(0,T)\times\T$ in vorticity formulation:
\begin{equation}\label{eq:e}
\begin{cases}
\partial_t \om+u\cdot\nabla\om=0,\\
u=\nabla^{\perp}(-\Delta)^{-1}\omega.
\end{cases}
\end{equation}
We assume periodic boundary conditions and the following initial condition for \eqref{eq:e} 
\begin{equation}\label{eq:ei}
\om|_{t=0}=\om_0.
\end{equation}
Next, let $\nu>0$ and consider the  $2D$ Navier-Stokes equations in $(0,T)\times\T$,
\begin{equation}\label{eq:ns}
\begin{cases}
\partial_t \omv+\uv\cdot\nabla\omv-\nu\Delta\omv=0,\\
\uv=\nabla^{\perp}(-\Delta)^{-1}\omv,
\end{cases}
\end{equation}
with initial datum 
\begin{equation}\label{eq:e2}
\omv|_{t=0}=\om_0^\nu
\end{equation}
and periodic boundary conditions.\\

We introduce the {\em Lagrangian formulations} of the systems \eqref{eq:e} and \eqref{eq:ns}.  We start with the Euler equations. We recall that for smooth solutions, by the theory of characteristics, if $X:[0,T]\times[0,T]\times \T\to\T$ solves 
\begin{align}
&\begin{cases}
\partial_s X_{t,s}(x)=u(s,X_{t,s}(x)),\hspace{0.5cm} s\in[0,T],\\
X_{t,t}(x)=x,
\end{cases}\label{eq:el1}
\end{align}
for any given $t\in(0,T)$, then 
\begin{align}
&u(t,x):=(\nabla^{\perp}(-\Delta)^{-1}\omega(t,\cdot))(x),\label{eq:el2}\\
&\omega(t,x):=\omega_0(X_{t,0}(x))\label{eq:el3}
\end{align}
solve the $2D$ Euler equations in $(0,T)\times\T$ with initial datum $\omega_0$.\\

Before introducing Lagrangian solutions to the Euler equations, we give the definition of flow of a non-smooth vector field.
\begin{defn}[Regular Lagrangian flows]\label{def:rlf}
The map $X\in L^{\infty}((0,T)\times(0,T)\times\T)$ is a regular Lagrangian flow of \eqref{eq:el1} if for a.e. $x\in \T$ and for any $t\in[0,T]$ the map $s\in[0,T]\mapsto X_{t,s}(x)\in\T$ is an absolutely continuous solution of \eqref{eq:el1} and for any $t\in[0,T]$ and $s\in[0,T]$ the map $x\in\T\mapsto X_{t,s}(x)\in\T$ is measure-preserving. 
\end{defn}
The definition of Lagrangian solutions of the Euler equations is the following:
\begin{defn}[Lagrangian solutions of the 2D Euler equations]\label{def:el}
Let $p\in(1,\infty)$ and $\omega_0\in L^{p}(\T)$. We say that $(u,\omega)$ is a {\em Lagrangian} solution of the $2D$ Euler equations if 
\begin{equation}\label{eq:reg}
(u,\omega)\in L^{\infty}((0,T);W^{1,p}(\T))\times L^{\infty}((0,T);L^{p}(\T)),
\end{equation}	
there exists a regular Lagrangian flow $X\in L^{\infty}((0,T)\times(0,T)\times\T)$ in the sense of Definition \ref{def:rlf}, and for a.e. $(t,x)\in(0,T)\times\T$ the functions $u$ and $\omega$ satisfy \eqref{eq:el2} and \eqref{eq:el3}. 
\end{defn}
We remark that in the regularity class \eqref{eq:reg}, given $u$ the regular Lagrangian flow $X$ is unique, see \cite{CDL} and \cite{DPL}. \\

Next, we consider the Navier-Stokes equations \eqref{eq:ns} and we recall that in two dimensions solutions of the Navier-Stokes equations \eqref{eq:ns} are regular and unique. Given a probability space $(\Omega, \mathcal{F}, \mathbb{P})$ we define the map $X^\nu:[0,T]\times[0,T]\times \T\times\Omega\to\T$ as follows.\par
 For $\mathbb{P}$-a.e. $\xi\in\Omega$ and for any $t\in(0,T)$, for $s\in[0,T]$ we consider a $\T$-valued Brownian motion $W_s$ adapted to the backward filtration, i.e. satisfying $W_t=0$. The map $s\mapsto X^{\nu}_{t,s}(x,\xi)$ is obtained by solving 
\begin{align}
&\begin{cases}
\de X^\nu_{t,s}(x,\xi)=u^\nu(s,X^\nu_{t,s}(x,\xi))\de s + \sqrt{2\nu} \de W_s(\xi),\hspace{0.5cm}s\in[0,t),\\
X^\nu_{t,t}(x,\xi)=x,
\end{cases}\label{eq:nsl1}
\end{align}
For $\mathbb{P}$-a.e. $\xi\in\Omega$ the map $x\in\T\mapsto X^{\nu}_{t,s}(x,\xi)\in\T$ is measure-preserving for any $t\in[0,T]$ (see \cite{LBL}) and $s\in[0,t]$ and, by the Feynman-Kac formula (see \cite{K, LBL}),  $\omv=\mathbb{E}[\omv_0(X_{t,0}(x))]$ solves the advection-diffusion equation
\begin{equation*}
\partial_{t}\omv+\uv\cdot\nabla\omv-\nu\Delta\omv=0
\end{equation*}
with initial datum $\omv_0$, where we have denoted by $\mathbb{E}[f]$ the average with respect to $\mathbb{P}$, also called expectation. Therefore,
\begin{align}
&\,\,\uv(t,x):=(\nabla^{\perp}(-\Delta)^{-1}\omv(t,\cdot))(x),\label{eq:nsl2}\\
&\,\,\omv(t,x):=\mathbb{E}[\omega^{\nu}_0(X^{\nu}_{t,0}(x))]\label{eq:nsl3}
\end{align}
solve the Navier-Stokes equations \eqref{eq:ns}.\par
 
 We remark that the probability space and the Brownian motion can be arbitrarily chosen. Indeed, since $\uv$ is a smooth function, the equation \eqref{eq:nsl1} is satisfied in the strong sense \cite{K}, see also \cite{LBL}, namely one can find a solution $X^{\nu}_{t,\cdot}$ to \eqref{eq:nsl1} on any given filtered probability space with any given adapted Brownian motions as described above.  

\subsection{Quantitative strong convergence of the vorticity}

In this section we prove our first main result.
\begin{thm}\label{teo:main2}
Let $p\in(1,\infty)$ and $\omega_0\in L^{p}(\T)$. Let $\{\omega_0^\nu\}_{\nu}\subset L^{p}(\T)$ be a sequence of smooth functions such that 
\begin{equation*}
\begin{aligned}
&\omega_0^\nu\to \omega_0\quad\mbox{ strongly in }L^{p}(\T),
\end{aligned}
\end{equation*}
and $(\uv,\omv)$ be the solution of the Navier-Stokes equations with initial datum $\omega_0^{\nu}$. 
Assume that there exists $(u,\omega)$ Lagrangian solution of the Euler equations such that, up to a subsequence not relabelled, 
\begin{equation}\label{eq:convl}
\begin{aligned}
&\uv\weaktos u\quad\mbox{ weakly* in }L^{\infty}((0,T);L^{2}(\T)).
\end{aligned}
\end{equation}
Then
\begin{equation*}
\begin{aligned}
\omv\to \omega\quad\mbox{ strongly in }C([0,T];L^{p}(\T)).
\end{aligned}
\end{equation*}
Moreover, for any $\delta>0$ there exists $C=C(\delta,\omega_0)>0$ such that for $\nu$ small enough
\begin{equation}\label{eq:quant}
\sup_{t\in(0,T)}\|\omv(t)-\om(t)\|_{L^{p}}\leq \delta+\frac{C(\delta,\omega_0)}{|\ln(\max\{\sqrt{\nu}, \|\uv-u\|_{L^{1}(L^1)}\})|}+\|\omega_0^\nu-\omega_0\|_{L^{p}}.
\end{equation}
\end{thm}
\begin{rem}\label{rem:equi}
The assumption that $(u,\omega)$ is Lagrangian is not restrictive. Indeed, if $p\geq 2$ every distributional solution of the Euler equations is renormalized \cite{FMN} and if $p\in[1,2)$ every solutions obtained as a limit of vanishing viscosity is renormalized \cite{CS, CNSS}. Moreover, the uniqueness of the linear problem (\cite{DPL,CNSS}) implies that every renormalized solution is Lagrangian.  
\end{rem}
\begin{rem}
We note that the solution $(u,\omega)$ satisfies the following conservations
\begin{equation*}
\begin{aligned}
&\|\om(t)\|_{L^p}=\|\om_0\|_{L^p},\quad &\|u(t)\|_{L^2}=\|u_0\|_{L^2},
\end{aligned}
\end{equation*}
where $u_0=\nabla^{\perp}(-\Delta)^{-1}\om_0$. Indeed, the conservation of the $L^p$-norm of the vorticity is a consequence of \eqref{eq:el1} and the fact that the flow $X_{t,0}(\cdot)$ is measure-preserving, while the conservation of the energy is one of the main results in \cite{CFLS}.
\end{rem}
\begin{rem}
Regarding the case $p=1$ we first notice that by following the same arguments in \cite{BC} and \cite{BBC2} we expect the strong convergence of the vorticity to hold in $C(L^1)$ by using the Lagrangian approach. The reason we did not include in Theorem \ref{teo:main2} the case $p=1$ is that we do not want to introduce more technical tools from Harmonic Analysis and we prefer to deal with this case with the Eulerian approach in the next section.  
\end{rem}

\begin{proof}[Proof of Theorem \ref{teo:main2}]
We divide the proof in several steps.\\
\\
\underline{Step 1}\hspace{0.5cm}{\em Weak convergence of the vorticity.}\\
\\
We prove that 
\begin{equation}\label{eq:weakconvom}
\omv\weaktos\om\quad\mbox{ weakly* in }L^{\infty}((0,T);L^{p}(\T)). 
\end{equation}
Since $(\uv,\omv)$ solves \eqref{eq:ns}-\eqref{eq:e2}, by standard $L^{p}$-estimates for the advection-diffusion equation satisfied by $\omv$ we have that 
\begin{equation}\label{eq:apriori}
\{\omv\}_{\nu}\mbox{ is bounded in }L^{\infty}((0,T);L^{p}(\T)).
\end{equation}
Since $p\in(1,\infty)$ there exists $\bar{\omega}\in L^{\infty}((0,T);L^{p}(\T))$ such that 
\begin{equation*}
\omv\weaktos\bar{\om}\quad\mbox{ weakly* in }L^{\infty}((0,T);L^{p}(\T)),
\end{equation*}  
and, by using \eqref{eq:convl}, 
\begin{equation*}
\curl\,u=\bar{\omega}
\end{equation*}
in the sense of distribution. Since $\curl\,u=\omega$ in the sense of distributions we conclude that $\omega=\bar{\omega}$ and \eqref{eq:weakconvom} is proved.\\
\\
\underline{Step 2}\hspace{0.5cm}{\em Strong convergence of the velocity.}\\
\\
We start by noticing that by \eqref{eq:convl} we have that 
\begin{equation}\label{eq:apriori}
\begin{aligned}
&\{\uv\}_{\nu}\mbox{ is bounded in }L^{\infty}((0,T);L^{2}(\T)),\\
&\{\omv\}_{\nu}\mbox{ is bounded in }L^{\infty}((0,T);L^{p}(\T)).
\end{aligned}
\end{equation}
 By Calder\'on-Zygmund theorem we have that 
\begin{equation*}
\|\nabla\uv(t)\|_{L^p}\leq C\|\omv(t)\|_{L^p}.
\end{equation*}
Therefore, by \eqref{eq:apriori}, we get that 
\begin{equation}\label{eq:boundnabla}
\{\nabla\uv\}_{\nu}\mbox{ is bounded in }L^{\infty}((0,T);L^{p}(\T)).
\end{equation}
Next, since $(\uv,\omv)$ solve the Navier-Stokes equations in the classical sense, we also have that 
\begin{equation*}\label{eq:nsvel}
\partial_t\uv+\dive(\uv\otimes\uv)-\nu\Delta\uv+\nabla p^{\nu}=0
\end{equation*}
where $p^{\nu}$ has zero-average and solves 
\begin{equation*}
-\Delta p^{\nu}=\dive(\dive(\uv\otimes\uv)).
\end{equation*}
Therefore, by using \eqref{eq:apriori} and \eqref{eq:boundnabla}, we have that for some $s$ large enough
\begin{equation*}
\{\partial_t\uv\}_{\nu}\mbox{ is bounded in }L^{\infty}((0,T);H^{-s}(\T)).
\end{equation*}
Then, by Aubin-Lions lemma we conclude that 
\begin{equation*}
\uv\to u\quad\mbox{ strongly in }C([0,T];L^{2}(\T)).
\end{equation*}
\\
\underline{Step 3}\hspace{0.5cm}{\em Comparison of the flows.}\\
\\
Since $(u,\om)$ is Lagrangian, by Definition \ref{def:el} there exists a regular Lagrangian flow $X$. Then, for $\mathbb{P}$-a.e. $\xi\in\Omega$, for a.e. $x\in\T$, and for any fixed $t\in(0,T)$ the following S.D.E. is satisfied for   
$s\in[0,t]$:
\begin{equation}\label{eq:difference}
\begin{cases}
\de(X^{\nu}_{t,s}(x,\xi)-X_{t,s}(x))=(\uv(s, X^{\nu}_{t,s}(x,\xi))-u(s,X_{t,s}(x)))\,\de s+\sqrt{2\nu}\de W_{s}(\xi),\\
X^{\nu}_{t,t}(x,\xi)-X_{t,t}(x)=0.
\end{cases}
\end{equation}
We define the function $q_{\e}(y)=\ln\left(1+\frac{|y|^2}{\e^2}\right)$ and the related functional
$Q^\e_\nu(t,s)$ as
$$
Q^\e_\nu(t,s):=q_{\e}(X^\nu_{t,s}-X_{t,s})=\ln\left(1+\frac{|X^\nu_{t,s}-X_{t,s}|^2}{\e^2}\right),
$$
where $\e>0$ is a fixed parameter that will be chosen later and we have omitted the explicit dependence on $x\in\T$ and $\xi\in\Omega$. By using It\^o's formula we get that
\begin{align*}
\int_{\T} \E\left[Q^\e_\nu(t,s)\right]\de x&=\int_s^t\int_{\T} \E\left[\nabla_{y} q_{\e}(X^\nu_{t,\tau}-X_{t,\tau})\cdot\left(u^\nu(s,X^\nu_{t,\tau})-u(\tau,X_{t,\tau}) \right)\right]\de x\de \tau\\ 
&+\nu\int_s^t\int_{\T} \E\left[\nabla_{y}^2 q_{\e}(X^\nu_{t,\tau}-X_{t,\tau})\right]\de x\de \tau,
\end{align*}
and from the inequalities
$$
\left| \nabla \ln\left(1+\frac{|y|^2}{\e^2} \right)\right|\leq\frac{C}{\e+|y|}, \hspace{0.7cm}\left| \nabla^2 \ln\left(1+\frac{|y|^2}{\e^2}\right)\right|\leq\frac{C}{\e^2+|y|^2},
$$
we obtain that
\begin{equation}\label{eq:219}
\int_{\T} \E\left[Q^\e_\nu(t,s)\right]\de x\leq \frac{C\nu (t-s)}{\e^2}+C\int_s^t\int_{\T} \E\left[\frac{\left| u^\nu(s,X^\nu_{t,\tau})-u(\tau,X_{t,\tau}) \right|}{\e+\left| X^\nu_{t,\tau}-X_{t,\tau}\right|}\right]\de x\de \tau.
\end{equation}
After adding and subtracting $u(s,X^\nu_{t,\tau})$ in the numerator of the second term on the right hand side of \eqref{eq:219} we estimate the resulting terms as follows:
\begin{equation}
\int_s^t\int_{\T} \E\left[\frac{\left| u^\nu(s,X^\nu_{t,\tau})-u(s,X^\nu_{t,\tau}) \right|}{\e+\left| X^\nu_{t,\tau}-X(s,x)\right|}\right]\de x\de \tau\leq \frac{1}{\e}\int_s^t\int_{\T}\mathbb{E}\left[\left| u^\nu(s,X^\nu_{t,\tau})-u(s,X^\nu_{t,\tau}) \right|\right] \de x \de \tau,
\end{equation}
\begin{equation}\label{eq:max1}
\begin{aligned}
\int_s^t\int_{\T} \E\left[\frac{\left| u(s,X^\nu_{t,\tau})-u(\tau,X_{t,\tau}) \right|}{\e+\left| X^\nu_{t,\tau}-X_{t,\tau}\right|}\right]\de x\de \tau&\leq C\int_s^t\int_{\T} \E\left[\mathcal{M}|\nabla u|(s,X^\nu_{t,\tau})\right]\de x \de \tau\\
&+C\int_s^t\int_{\T} \mathcal{M}|\nabla u|(s,X_{t,\tau})\de x \de \tau.
\end{aligned}
\end{equation}
Above, we have used the following maximal inequality 
\begin{equation*}
|u(s,x)-u(s,y)|\leq C\mathsf{d}(x,y)\left(\mathcal{M}|\nabla\,u|(s,\cdot)(x)+\mathcal{M}|\nabla\,u|(s,\cdot)(y)\right),\end{equation*}
 for a.e. $x,y\in\T$ and $s\in(0,T)$. To estimate the right-hand side of \eqref{eq:max1} we use that  $X_{t,\tau}$ and $X^\nu_{t,\tau}$ are measure preserving, that $\mathsf{d}(x,y)\leq |x-y|$, that the maximal function operator is continuous on $L^{q}(\T)$ for $q>1$, and that $\T$ has finite measure.  In this way we obtain 
\begin{equation*}
\int_s^t\int_{\T} \E\left[\frac{\left| u(s,X^\nu_{t,\tau})-u(\tau,X_{t,\tau}) \right|}{\e+\left| X^\nu_{t,\tau}-X_{t,\tau}\right|}\right]\de x\de \tau\leq C\|\nabla\,\uv\|_{L^{1}(L^q)}.
\end{equation*}
In conclusion, we finally get that
\begin{equation}
\int_{\T} \E\left[Q^\e_\nu(t,s)\right]\de x\leq C\left( \frac{\nu (t-s)}{\e^2}+\frac{1}{\e}\|u^\nu-u\|_{L^1(L^1)}+\|\nabla u\|_{L^1(L^q)}\right).
\end{equation}
Next, note that 
\begin{align}
\label{est:stab}
\left(\mathscr{L}^2\otimes\mathbb{P}\right)&\left(\{(x,\xi)\in \T\times\Omega:\mathsf{d}(X^\nu_{t,s}(x,\xi),X_{t,s}(x))> \sqrt{\e}\}\right)\nonumber\\ &\leq \frac{C}{|\ln\e|}\int_{\T} \E\left[\ln\left(1+\frac{(\mathsf{d}(X^\nu_{t,s}(x,\xi),X_{t,s}(x)))^2}{\e^2}\right)\right]\de x\\
\nonumber&\leq \frac{C}{|\ln\e|}\int_{\T} \E\left[Q^\e_\nu(t,s)\right]\de x\\
\nonumber&\leq C\left( \frac{\nu (t-s)}{\e^2|\ln\e|}+\frac{1}{\e|\ln \e|}\|u^\nu-u\|_{L^1(L^1)}+\frac{1}{|\ln\e|}\|\nabla u\|_{L^1(L^q)}\right),
\end{align}
where we have used that $x,y\in\T$, $\mathsf{d}(x,y)\leq |x-y|$, and that the function  $z\to\log(1+\frac{z^2}{\e^2})$ is increasing on $[0,\infty)$. Therefore, 
\begin{align}
\label{eq:split}
\int_{\T}\mathbb{E}[\mathsf{d}(X^\nu_{t,s}(x,\xi),X_{t,s}(x))]\de x &=\int_{\{(x,\xi)\in\T\times\Omega: \,\,\mathsf{d}(X^\nu_{t,s}(x,\xi),X_{t,s}(x))\leq \sqrt{\e}\}}\mathsf{d}(X^\nu_{t,s}(x,\xi),X_{t,s}(x))\de \mathbb{P}\de x\nonumber\\
&+\int_{\{(x,\xi)\in\T\times\Omega:\,\, \mathsf{d}(X^\nu_{t,s}(x,\xi),X_{t,s}(x))>\sqrt{\e}\}}\mathsf{d}(X^\nu_{t,s}(x,\xi),X_{t,s}(x))\de \mathbb{P}\de x\\
\nonumber&\lesssim \sqrt{\e}+\left(\mathscr{L}^2\otimes\mathbb{P}\right)\left(\{(x,\xi)\in \T\times\Omega:\mathsf{d}(X^\nu_{t,s}(x,\xi),X_{t,s}(x))> \sqrt{\e}\}\right),
\end{align}
where we have used that $\mathscr{L}^2 \otimes\mathbb{P}$ is a probability measure on $\T\times\Omega$ and that  the distance $\mathsf{d}$ on the torus is bounded. We first choose  $\e=\e(\nu):=\max\{\sqrt{\nu},\|u^\nu-u\|_{L^1(L^1)}\}$ and we use \eqref{est:stab} in \eqref{eq:split}. Noticing that there exists $\nu_0>0$ such that for every $\nu\leq \nu_0$ it holds $\sqrt{\e(\nu)}\leq \frac{1}{|\ln \e(\nu)|}$ we conclude that  
\begin{equation}
\label{eq:exp_dist}
\int_{\T}\mathbb{E}[\mathsf{d}(X^\nu_{t,s}(x,\xi),X_{t,s}(x))]\de x\leq \sqrt{\e(\nu)}+C\frac{(t-s)+1}{|\ln\e(\nu)|}\leq \frac{CT}{|\ln\e(\nu)|}.
\end{equation}

\noindent\underline{Step 4}\hspace{0.5cm}{\em Strong convergence of the vorticity.}\\
\\
Let $n\in\N$ and $\{\omega_0^{n}\}_{n}$ be a sequence of Lipschitz approximations of $\omega_0$. 
For any $t\in(0,T)$, by using Jensen's inequality, we have that
\begin{equation*}
\begin{aligned}
\|\omv(t)-\om(t)\|_{L^p}&=\|\mathbb{E}[\omv_0(X_{t,0}^{\nu})]-\om_0(X_{t,0})\|_{L^p}\\
&\leq\left(\int_{\T}\int_\Omega |\omv_0(X^\nu_{t,0})-\om_0(X^\nu_{t,0})|^p\de\mathbb{P}\de x\right)^{\frac{1}{p}}\\
&+\left(\int_{\T}\int_\Omega |\om_0^n(X^\nu_{t,0})-\om_0(X^\nu_{t,0})|^p\de\mathbb{P}\de x\right)^{\frac{1}{p}}\\
&+\left(\int_{\T}|\om_0^n(X_{t,0})-\om_0(X_{t,0})|^p\de x\right)^{\frac{1}{p}}\\
&+\left(\int_{\T}\int_\Omega |\om_0^n(X^\nu_{t,0})-\om_0^n(X_{t,0})|^p\de\mathbb{P}\de x\right)^{\frac{1}{p}}.
\end{aligned}
\end{equation*}
In particular, by using \eqref{eq:exp_dist} and that $\omega_0^n$ is Lipschitz we have 
\begin{equation*}
\begin{aligned}
\|\mathbb{E}[|\om^{n}_0(X_{t,0}^{\nu})-\om^{n}_0(X_{t,0})|]\|_{L^p}^p&\leq C_n\|\mathbb{E}[\mathsf{d}(X^\nu_{t,0},X_{t,0})]\|_{L^p}^{p}\\
&\leq \frac{C_n}{|\ln(\max\{\sqrt{\nu},\|\uv-u\|_{L^{1}(L^1)}\})|^p},
\end{aligned}
\end{equation*}
and then we get 
\begin{equation*}
\|\omv(t)-\om(t)\|_{L^p}\leq\|\omega_0^\nu-\omega_0\|_{L^p}+2\|\omega_0^n-\omega_0\|_{L^p}+\frac{C_n}{|\ln(\max\{\sqrt{\nu},\|\uv-u\|_{L^{1}(L^1)}\})|}.
\end{equation*}
Then, since $\uv$ converges to $u$ in $L^1((0,T);L^1(\T))$, sending first $\nu\to 0$ and then $n\to\infty$ it follows that $\omv\to \om$ strongly in $C([0,T];L^{p}(\T))$. The quantitative estimate \eqref{eq:quant} follows as well. 
\end{proof}

\subsection{Rate of covergence for bounded vorticity}
In this subsection we study the rate of convergence for bounded vorticity. We first recall the following result of J.-Y. Chemin \cite{Ch}.  
\begin{thm}\label{teo:chemin}
Let $\omega_0\in L^{\infty}(\T)$ and set $M:=\|\omega_0\|_{L^\infty}$. Let $(u,\omega)$ and $(\uv,\omv)$ be the unique  solutions on $(0,T)\times\T$ of the Euler and Navier-Stokes equations with the same initial datum $\omega_0$. Then, there exist $\nu_0=\nu_0(T, M)$ and $C=C(T, M)$ such that 
for any $\nu\leq \nu_0$ 
\begin{equation}\label{eq:chemin}
\sup_{t\in(0,T)}\|\uv(t)-u(t)\|_{L^2}\leq C\nu^{\frac{e^{-CT}}{2}}=:\delta_{\nu}^{M,T}.
\end{equation}
\end{thm}
We remark that in \cite{Ch} the theorem is stated and proved when the domain is the entire space $\R^2$. The proof in \cite{Ch} works also in the case of the torus with minor changes. Notice that also a different proof of Theorem \ref{teo:chemin} is given in \cite[Lemma 4]{CDE} and a $log$-improvement of the rate has been obtained in \cite{Se}. 
The main theorem of this subsection is the following.
\begin{thm}\label{teo:rate}
Let $\omega_0\in L^{\infty}(\T)$ and set $M:=\|\omega_0\|_{\infty}$. Let $(u,\omega)$ and $(\uv,\omv)$ be the unique bounded solutions on $(0,T)\times\T$ of the Euler and Navier-Stokes equations with the same initial datum $\omega_0$. Then, there exists $\nu_0=\nu_0(T, M,\omega_0)$ and a continuous function $\phi_{\omega_0,p,M}:\R^+\to\R^+$ with $\phi_{\omega_0,p,M}(0)=0$, such that for any $1\leq p< \infty$
\begin{equation}\label{eq:rate}
\sup_{t\in (0,T)}\|\omv(t)-\om(t)\|_{L^p}\leq CM^{1-\frac{1}{p}}\max\{\phi_{\omega_0,p,M}(\delta^{M,T}_{\nu}), (\delta_{\nu}^{M,T})^{\frac{e^{-CT}}{2p}}\},
\end{equation}
where $\delta_{\nu}^{M,T}$ is defined in \eqref{eq:chemin}.
\end{thm}
Before giving the proof of Theorem \ref{teo:rate} we recall the following version of  Osgood lemma, see \cite{Ch}.
\begin{lem}\label{lem:osgood}
Let $\rho$ be a positive Borel function, $\gamma$ a locally integrable positive function, and $\mu$ a continuous increasing function. Assume that, for some strictly positive number $\alpha$, the function $\rho$ satisfies
$$
\rho(t)\leq \alpha+\int_t^{t_0}\gamma(s)\mu(\rho(s))\de s.
$$
Then we have that
$$
-\mathfrak{M}(\rho(t))+\mathfrak{M}(\alpha)\leq \int_t^{t_0}\gamma(s)\de s, \hspace{0.3cm}\mbox{  with }\mathfrak{M}(x)=\int_x^1\frac{1}{\mu(s)}\de s.
$$
\end{lem}
\begin{proof}[Proof of Theorem \ref{teo:rate}]
We divide the proof in several steps.\\
\\
\underline{Step 1} \hspace{0.5cm}{\em Rate on the difference of the flows.}\\
\\
Let $X^\nu_{t,s}, X_{t,s}$ be respectively the solutions of \eqref{eq:nsl1} and \eqref{eq:e}. By It\^o's formula we have that
\begin{equation}\label{eq:ito2}
\begin{aligned}
\frac{|X^\nu_{t,s}-X_{t,s}|^2}{2}&=\int_s^t [(u^\nu(\tau,X^\nu_{t,\tau})-u(\tau,X_{t,\tau}))\cdot (X^\nu_{t,\tau}-X_{t,\tau})+2\nu]\de \tau\\&+\sqrt{2\nu}\int_s^t (X^\nu_{t,\tau}-X_{t,\tau})\cdot\de W_\tau.
\end{aligned}
\end{equation}
Next, we have the simple estimate
\begin{align*}
|(u^\nu(\tau,X^\nu_{t,\tau})-u(\tau,X_{t,\tau}))\cdot (X^\nu_{t,\tau}-X_{t,\tau})|\leq& |u^\nu(\tau,X^\nu_{t,\tau})-u(\tau,X^\nu_{t,\tau})||X^\nu_{t,\tau}-X_{t,\tau}|\\
&+|u(\tau,X^\nu_{t,\tau})-u(\tau,X_{t,\tau})||X^\nu_{t,\tau}-X_{t,\tau}|\\
\leq& \frac{|u^\nu(\tau,X^\nu_{t,\tau})-u(\tau,X^\nu_{t,\tau})|^2}{2}+\frac{|X^\nu_{t,\tau}-X_{t,\tau}|^2}{2}\\
&+ C \frac{|X^\nu_{t,\tau}-X_{t,\tau}|^2}{2}\left(\mathcal{M}|\nabla u(\tau,\cdot)|( X^\nu_{t,\tau})+ \mathcal{M}|\nabla u(\tau,\cdot)|( X_{t,\tau})\right).
\end{align*}
Then, taking the expected value and integrating in space, we can estimate \eqref{eq:ito2} as follows
\begin{align*}
\int_{\T}\int_{\Omega}&\frac{|X^\nu_{t,s}-X_{t,s}|^2}{2}\de \mathbb{P}\de x \\ &\leq 2\nu(t-s)+\int_s^t\int_{\T}\int_{\Omega}\frac{|u^\nu(\tau,X^\nu_{t,\tau})-u(\tau,X^\nu_{t,\tau})|^2}{2}\de \mathbb{P}\de x\de \tau\\
&+\int_s^t \left(\int_\Omega\int_{\T} \mathcal{M}|\nabla u(\tau,\cdot)|( X^\nu_{t,\tau})^p\de x\de \mathbb{P} \right)^{\frac{1}{p}}\left( \int_{\T}\int_{\Omega}|X^\nu_{t,\tau}-X_{t,\tau}|^{\frac{2p}{p-1}}\de \mathbb{P}\de x \right)^{\frac{p-1}{p}} \de \tau\\
&+\int_s^t\left(\int_{\T} \mathcal{M}|\nabla u(\tau,\cdot)|( X_{t,\tau})^p\de x \right)^{\frac{1}{p}}\left(\int_{\T}\int_{\Omega}|X^\nu_{t,\tau}-X_{t,\tau}|^{\frac{2p}{p-1}} \de \mathbb{P}\de x\right)^{\frac{p-1}{p}}\de \tau\\
&+ \int_s^t\int_{\T}\int_{\Omega} \frac{|X^\nu_{t,\tau}-X_{t,\tau}|^2}{2} \de \mathbb{P}\de x\de \tau.
\end{align*}
We recall that by Calder\'on-Zygmund theorem we have that for $p<\infty$ large
\begin{equation}\label{eq:constcz}
\|\nabla\uv(t)\|_{L^p}\leq C\,p\|\omega(t)\|_{L^p}.
\end{equation}
Therefore, by using the measure-preserving property of $X^\nu_{t,s}$ and $X_{t,s}$, the boundedness of the flows and the fact that the maximal function is bounded in $L^{p}(\T)$ for any $1<p\leq\infty$, we obtain that
\begin{align*}
\int_{\T}\int_\Omega|X^\nu_{t,s}-X_{t,s}|^2\de \mathbb{P}\de x &\leq \left(4\nu+\|u^\nu-u\|_{L^\infty(L^2)}^2\right)(t-s)+\int_s^t\int_{\T}\int_\Omega|X^\nu_{t,\tau}-X_{t,\tau}|^2 \de \mathbb{P}\de x\de \tau\\
&+C\int_s^t\|\nabla u(\tau,\cdot)\|_{L^p}\left( \int_{\T}\int_\Omega|X^\nu_{t,\tau}-X_{t,\tau}|^2 \de \mathbb{P}\de x\right)^{\frac{p-1}{p}}\de \tau\\
&\leq \left(4\nu+\|u^\nu-u\|_{L^\infty(L^2)}^2\right)(t-s)+\int_s^t\int_{\T}\int_\Omega|X^\nu_{t,\tau}-X_{t,\tau}|^2 \de \mathbb{P}\de x\de \tau\\
&+C\,M\,p\int_s^t\left( \int_{\T}\int_\Omega|X^\nu_{t,\tau}-X_{t,\tau}|^2 \de \mathbb{P}\de x\right)^{\frac{p-1}{p}}\de \tau,
\end{align*}
where we have used \eqref{eq:constcz} and the bound in $L^{\infty}(\T)$ on the vorticity. Therefore, if we define
$$
y_\nu(t,s):=\int_{\T}\int_\Omega|X^\nu_{t,s}-X_{t,s}|^2\de \mathbb{P}\de x, \hspace{0.4cm}\alpha_\nu^T:=\left(4\nu+\|u^\nu-u\|_{L^\infty(L^2)}^2\right)T,
$$
for any $s,t\in(0,T)$ with $s<t$, we can rewrite the above estimate as
\begin{equation}
\begin{cases}
y_\nu(t,s)\leq \alpha_\nu^T+\displaystyle\int_s^t (y_\nu(t,\tau)+Cpy_\nu(t,\tau)^{1-\frac{1}{p}})\de \tau,\\
y_\nu(t,t)=0,
\end{cases}
\end{equation}
where the constant $C$ depends on $M$ and we have used \eqref{eq:constcz}. Moreover, by \eqref{eq:chemin} we can estimate
$$
\alpha_\nu^T\leq C\delta_{\nu}^{M,T},
$$
and we get that
\begin{equation}
\begin{cases}
y_\nu(t,s)\leq C\delta_{\nu}^{M,T}+\displaystyle\int_s^t (y_\nu(t,\tau)+Cpy_\nu(t,\tau)^{1-\frac{1}{p}})\de \tau,\\
y_\nu(t,t)=0.
\end{cases}
\end{equation}
At this point we can argue as in \cite{Ch}: we choose $p=2-\ln(y^\nu(t,\tau))$ and since we can assume $y_\nu<1$, we get that
\begin{align*}
y_\nu(t,s)&\leq C\delta_{\nu}^{M,T}+\displaystyle\int_s^t y_\nu(t,\tau)+C(2-\ln(y^\nu(t,\tau))y_\nu(t,\tau)^{1-\frac{1}{2-\ln(y^\nu(t,\tau))}}\de \tau\\
&\leq C\delta_{\nu}^{M,T}+C\displaystyle\int_s^t(2-\ln(y^\nu(t,\tau)))y_\nu(t,\tau)\de \tau.
\end{align*}
Then, by using Lemma \ref{lem:osgood} with
$$
\rho(s):=y_\nu(t,s),\hspace{0.5cm}\alpha:=C\delta_{\nu}^{M,T},\hspace{0.5cm}\gamma(x):=C
$$
$$
\mu(x):=x(2-\ln x),\hspace{0.5cm}\mathfrak{M}(x):=\ln(2-\ln x)-\ln 2,
$$
we obtain that
\begin{equation}
-\ln(2-\ln y_\nu(t,s))+\ln(2-\ln\delta_\nu^{M,T})\leq C(t-s) ,
\end{equation}
which implies that
\begin{equation}\label{eq:rate_y_nu}
y_\nu(t,s)\leq \exp\left(2-2e^{-c(t-s)}\right)\left(\delta_\nu^{M,T}\right)^{e^{-C(t-s)}}\leq C\left(\delta_\nu^{M,T}\right)^{e^{-CT}},
\end{equation}
or in other words
\begin{equation}\label{eq:rate_flow}
\int_{\T}\mathbb{E}[\mathsf{d}(X^\nu_{t,s},X_{t,s})^2]\leq\int_{\T}\mathbb{E}[|X^\nu_{t,s}-X_{t,s}|^2]\de x\leq C\left(\delta_\nu^{M,T}\right)^{e^{-CT}}.
\end{equation}
\\
\underline{Step 2}\hspace{0.5cm}{\em Rate of convergence of the vorticities.}\\
\\
Since $\omega_0\in L^\infty(\T)\subset L^1(\T)$, we can use the continuity of the translation operator in $L^1(\T)$ to infer that
there exist $h_0$ and a modulus of continuity $\phi_{\om_0,M}$ such that 
\begin{equation}\label{eq:tran_1}
\|\omega_0(\cdot+h)-\omega_0(\cdot)\|_{L^1}\leq \phi_{\om_0,M}(|h|)\quad\mbox{for }|h|\leq h_0.
\end{equation}
Then we get
\begin{align*}
\|\omega^\nu(t)-\omega(t)\|_{L^1}&=\int_{\T}|\omega^\nu(t,x)-\omega(t,x)|\de x =\int_{\T}|\mathbb{E}[\omega_0(X^\nu_{t,0})]-\omega_0(X_{t,0})|\de x\\
&\leq \iint_{\{\mathsf{d}(X^\nu_{t,0},X_{t,0})\leq\e\}}|\omega_0(X^\nu_{t,0})-\omega_0(X_{t,0})|\de \mathbb{P}\de x\\
&+\iint_{\{\mathsf{d}(X^\nu_{t,0},X_{t,0})>\e\}}|\omega_0(X^\nu_{t,0})-\omega_0(X_{t,0})| \de \mathbb{P}\de x\\
&\leq \phi_{\omega_0,M}(\e)+\frac{2M}{\e^2}\int_{\T}\mathbb{E}[\mathsf{d}(X^\nu_{t,0},X_{t,0})^2]\de x\\
&\leq \phi_{\omega_0,M}(\e)+\frac{C}{\e^2}\left(\delta_\nu^{M,T}\right)^{e^{-CT}},
\end{align*}
where in the last two inequalities we have used \eqref{eq:tran_1} and then \eqref{eq:rate_flow}. Finally, to get \eqref{eq:rate} it is enough to choose
$$
\e(\nu)=\left(\delta_\nu^{M,T}\right)^{\frac{e^{-CT}}{4}},
$$
to take $\nu_0$ such that $\e(\nu)\leq h_0$ for $\nu\leq \nu_0$ and finally to interpolate $L^p$ between $L^1$ and $L^\infty$.
\end{proof}

\section{The Eulerian approach}
The section is organized as follows: first we recall the definition of renormalized solutions of the Euler equations. Then we prove some preliminary lemmas and finally we show the main result (Theorem \ref{teo:strong_convergence_vort}). 

\subsection{Renormalized solutions of the $2D$ Euler equations and main result}
We consider the Cauchy problem for the $2D$ Euler equations in $(0,T)\times\R^2$:
\begin{equation}\label{eq:er}
\begin{cases}
\partial_t \om+u\cdot\nabla\om=0,\\
u=K\ast\omega,\\
\om|_{t=0}=\om_0,
\end{cases}
\end{equation}
where $K:\R^2\to\R^2$ is the Biot-Savart kernel given by $K(x)=\displaystyle\frac{1}{2\pi}\frac{x^{\perp}}{|x|^2}$.\\ 
Next, let $\nu>0$ and consider the Cauchy problem for the $2D$ Navier-Stokes equations in $(0,T)\times\R^2$,
\begin{equation}\label{eq:nsr}
\begin{cases}
\partial_t \omv+\uv\cdot\nabla\omv-\nu\Delta\omv=0,\\
\uv=K\ast\omv,\\
\omv|_{t=0}=\om_0^\nu.
\end{cases}
\end{equation}
Renormalized solutions for the system \eqref{eq:er} are defined in analogy with the ones introduced by DiPerna-Lions \cite{DPL} for the linear transport equations.  
\begin{defn}[Renormalized solutions of the $2D$ Euler equations]\label{def:RSvort}
Let $\omega_0\in L_{c}^{p}(\R^2)$ and $\omega\in C([0,T];L^p(\R^2))$ with $1\leq p< \infty$. The pair $(u,\omega)$ is a renormalized solution of \eqref{eq:er} if for any $\beta\in C^1(\R)\cap L^\infty(\R)$ vanishing in a neighbourhood of zero it holds
\begin{equation}\label{eq:ren}
\int_0^T \int_{\R^2} \beta(\omega)(\partial_t\varphi+u\cdot\nabla\varphi) \de x \de t+\int_{\R^2} \beta(\omega_0)\varphi(0,x) \de x=0,
\end{equation}
for any $\varphi\in C^\infty_c([0,T)\times\R^2)$, and 
$$
u(t,x)=(K*\omega(t,\cdot))(x)\hspace{0.7cm}\mbox{a.e. in }(0,T)\times\R^2.
$$
\end{defn}
Note that if  $\omega\in C([0,T];L^p(\R^2))$  and $\beta$ is as in Definition \ref{def:RSvort} then the composition $\beta(\omega)\in L^\infty((0,T);L^1(\R^2)\cap L^\infty(\R^2))$, therefore \eqref{eq:ren} makes sense. We remark that, in general, the vorticity equations cannot be interpreted in distributional sense if $1\leq p<4/3$. 
The main theorem of this section is the following.
\begin{thm}\label{teo:strong_convergence_vort}
Let $p\in[1,\infty)$ and $\omega_0\in L_{c}^{p}(\R^2)$. Let $\{\omega_0^\nu\}_{\nu}$ be a sequence of smooth compactly supported functions such that there exists $R>0$ with $\mathrm{supp}$ $\omega_0^{\nu}\subset B_{R}(0)$ and 
\begin{equation*}
\begin{aligned}
&\{\omega_0^\nu\}_{\nu}\mbox{ is bounded in }L^{p}(\R^2)\cap H^{-1}_{\mathrm{loc}}(\R^2),\\
&\omega_0^\nu\to \omega_0\quad\mbox{ strongly in }L^{p}(\R^2).
\end{aligned}
\end{equation*}
Let $(\uv,\omv)$ be the solution of the Navier-Stokes equations with initial datum $\omega_0^{\nu}$. 
Assume that there exists $(u,\omega)$ renormalized solution of the Euler equations such that 
\begin{equation}\label{eq:conv}
\begin{aligned}
&\uv\weaktos u\hspace{0.3cm}\mbox{ weakly* in }L^{\infty}((0,T);L^2_{\mathrm{loc}}(\R^2)).\\
\end{aligned}
\end{equation}
Then,  
\begin{equation*}
\begin{aligned}
\omv\to \omega\hspace{0.3cm}\mbox{ strongly in }C([0,T];L^{p}(\R^2)).
\end{aligned}
\end{equation*}
\end{thm}
\subsection{A preliminary lemma}
Let us consider the Cauchy problem for the linear transport equation
\begin{equation}
\label{eq:TE}
\begin{cases}
\partial_t \rho+b\cdot\nabla\rho=0,\\
\rho(0,\cdot)=\rho_0,
\end{cases}
\end{equation}
where $\rho_0:\R^d\to\R$ is a given initial datum in $L^{1}(\R^d)\cap L^{\infty}(\R^d)$ and $b:[0,T]\times\R^d\to \R^d$ is a given vector field satisfying the following assumptions:
\begin{itemize}
\item[(H1)] $b\in L^1((0,T);W^{1,p}_{\mathrm{loc}}(\R^d))$ for some $p>1$;
\item[(H1')] $b\in L^1((0,T);L^p_{\mathrm{loc}}(\R^d))$ for some $p>1$ and $\nabla b=S*g$ where $S:\R^d\to\R^{d\times d}$ is a singular integral operator of fundamental type \cite{Ste} and $g\in L^1((0,T)\times\R^d))$;
\item[(H2)] $b\in L^\infty((0,T);L^1(\R^d))+L^\infty((0,T)\times\R^d)$;
\item[(H3)] $\dive b=0$ in the sense of distributions.
\end{itemize}
Under the above hypothesis the transport equation \eqref{eq:TE} admits a unique solution in the class of densities $\rho\in L^\infty((0,T);L^1(\R^d)\cap L^\infty(\R^d))$, which is also renormalized, see \cite{BC, DPL}. Moreover, the velocity field $u$ of the two-dimensional Euler equations \eqref{eq:er} with vorticity $\omega\in L^\infty((0,T);L^1(\R^2)\cap L^p(\R^2))$ satisfies the above assumptions. Indeed, by the Biot-Savart law the gradient of the velocity field is a singular integral operator applied to the vorticity $\omega$, therefore the velocity field satisfies (H1) for $p>1$ and (H1') for $p=1$.\\

Let $\nu>0$ and consider a sequence $\{\rho^\nu\}_{\nu}$ of solutions of the following advection-diffusion equation with vector field $b^{\nu}$ and initial datum $\rho_0^\nu$
 \begin{equation}
\label{eq:TEe}
\begin{cases}
\partial_t \rho^\nu+b^{\nu}\cdot\nabla\rho^\nu=\nu\Delta\rho^\nu,\\
\rho^\nu(0,\cdot)=\rho^\nu_0.
\end{cases}
\end{equation}
We assume that 
\begin{equation}\label{eq:h2}
\{b^{\nu}\}_{\nu}\mbox{ is bounded in }L^\infty((0,T);L^1(\R^d))+L^\infty((0,T)\times\R^d)
\end{equation}
and for some $m>1$
\begin{equation}\label{eq:convb}
b^{\nu}\to b\quad\mbox{ strongly in }L^{m}_{\mathrm{loc}}((0,T)\times\R^d).
\end{equation}
To avoid technicalities we assume that $b^{\nu}$ is smooth. Moreover, we assume that $\{\rho^{\nu}_{0}\}_{\nu}$ is such that 
\begin{equation}\label{eq:convrho0}
\begin{aligned}
&\rho^{\nu}_{0}\to\rho_0\quad\mbox{ strongly in }L^{1}(\R^d),\\
&\rho^{\nu}_{0}\weaktos\rho_0\quad\mbox{ weakly* in }L^{\infty}(\R^d).
\end{aligned}
\end{equation}

The following lemma is a combination of Theorem IV.1 and Theorem II.4 in \cite{DPL}, generalized also to the case of vector fields satisfying (H1') instead of (H1).  
\begin{lem}
\label{teo:strong_convergence_TE}
Let $\rho_0\in L^{1}(\R^d)\cap L^{\infty}(\R^d)$ and $\{\rho^\nu_0\}_{\nu}$ satisfying \eqref{eq:convrho0}. Let $b$ be a vector field which satisfies (H1) or (H1'), (H2), and (H3) and let the smooth vector field $b^{\nu}$ satisfy \eqref{eq:h2} and \eqref{eq:convb}. Then, the unique solutions $\rho^\nu,\rho\in L^\infty((0,T);L^1(\R^d)\cap L^\infty(\R^d))$ of \eqref{eq:TE} and \eqref{eq:TEe} satisfy
$$
\rho^\nu\to\rho\quad \mbox{ in } C([0,T];L^q(\R^d)), \mbox{ for all } 1\leq q<\infty.
$$
\end{lem}
\begin{proof}
We divide the proof in several steps.\\
\\
\underline{Step 1}\hspace{0.5cm}\emph{Strong convergence in $L^q((0,T)\times\R^2)$ , $1<q<\infty$.}\\
\\
Let $\rho^\nu$ be the unique solution of \eqref{eq:TEe}. Then, for all $1\leq q\leq\infty$ we have that 
\begin{equation}
\label{est:rhonu}
\|\rho^\nu(t)\|_{L^q}\leq \|\rho_0^\nu\|_{L^q},
\end{equation}
and from \eqref{eq:convrho0} we deduce that $\rho^\nu$ is equi-bounded in $L^\infty((0,T);L^{1}(\R^d)\cap L^{\infty}(\R^d))$.
Then, up to a subsequence, there exists $\bar{\rho}\in L^\infty((0,T);L^{1}(\R^d)\cap L^{\infty}(\R^d))$ such that for any $1< q<\infty$
\begin{equation}\label{weakconvrp}
\rho^\nu\weakto\bar{\rho} \hspace{0.5cm}\mbox{ in } L^q((0,T)\times\R^d).
\end{equation}
Because of the linearity of the equation, it is immediate to deduce that $\bar{\rho}$ is a solution of \eqref{eq:TE} and by uniqueness it must be $\bar{\rho}=\rho$. Moreover, since $\rho$ is a renormalized solution of \eqref{eq:TE} it holds that
$$
\int_{\R^d}|\rho(t,x)|^q\de x=\int_{\R^d}|\rho_0(x)|^q\de x.
$$
By the lower semi-continuity of the $L^q$-norms with respect to the weak convergence we have that
\begin{align*}
\|\rho\|_{L^q(L^q)} & \leq \liminf_{\nu\to 0}\|\rho^\nu\|_{L^q(L^q)}\leq \limsup_{\nu\to 0}\|\rho^\nu\|_{L^q(L^q)}\\
& \leq T^{\frac{1}{q}} \lim_{\nu\to 0}\|\rho_0^\nu\|_{L^q}= T^{\frac{1}{q}}\| \rho_0\|_{L^q}=\|\rho\|_{L^q(L^q)},
\end{align*}
which implies the convergence of $\|\rho^\nu\|_{L^q(L^q)}$ towards $\|\rho\|_{L^q(L^q)}$. This latter fact, together with the weak convergence in \eqref{weakconvrp}, implies that
\begin{equation}\label{strongrp}
\rho^\nu\to\rho\hspace{0.5cm}\mbox{ in }L^q((0,T)\times\R^d),
\end{equation}
for all $1< q<\infty$. 
\\
\\
\underline{Step 2} \hspace{0.5cm}\emph{Convergence in $C([0,T];L^q_w(\R^d))$, $1<q<\infty$.}\\
\\
 By using the equation, it is a well-known fact that a weak solution $\rho$ of \eqref{eq:TE}, with initial datum $\rho_0\in L^q(\R^d)$, lies in the space $C([0,T];L^q_w(\R^d))$. In particular, this means that for any $\varphi\in C^\infty_c(\R^d)$ the map
$$
f_{\varphi}:t\in [0,T]\mapsto \int_{\R^d}\rho(t,x)\varphi(x)\,\de x
$$
is continuous. For any $\varphi\in C^\infty_c(\R^d)$ define the sequence of functions $f^\nu_\varphi$ as
$$
f_\varphi^\nu:t\in [0,T]\mapsto\int_{\R^d}\rho^\nu(t,x)\varphi(x)\de x.
$$
First of all, we have that
\begin{equation}
\sup_{t\in [0,T]} |f^\nu_\varphi(t)| =\sup_{t\in [0,T]} \left|\int_{\R^d}\rho^\nu(t,x)\varphi(x)\de x \right|\leq C \|\rho_0\|_{L^q}\|\varphi\|_{L^{q'}}.
\end{equation}
Moreover, by using the equation we have that
$$
\dot{f}^\nu_\varphi(t) =\int_{\R^d}\rho^\nu(t,x)b^{\nu}(t,x)\cdot\nabla\varphi(x)\de x+\nu\int_{\R^d}\rho^\nu(t,x)\Delta\varphi(x)\de x,
$$
which is uniformly bounded in $[0,T]$ by using \eqref{eq:h2} and \eqref{est:rhonu}. By Step 1, it follows that
$$
\dot{f}^\nu_\varphi\to\dot{f}_\varphi\hspace{0.5cm}\mbox{ in }L^1((0,T)),
$$
which eventually implies that
$$
f^\nu_\varphi\to f_\varphi \hspace{0.5cm}\mbox{ uniformly in }[0,T].
$$
By using the density of $C^\infty_c(\R^d)$ in $L^{q'}(\R^d)$, the previous convergence is equivalent to saying that
$$
\rho^\nu\to\rho\hspace{0.5cm}\mbox{ in }C([0,T];L^q_w(\R^d)).
$$
\\
\underline{Step 3}\hspace{0.5cm}\emph{Convergence of the $L^q$-norms on bounded sets.}\\
\\
Let $\beta\in L^{\infty}(\R)\cap C^2(\R)$ and define the functions
$$
f_{\beta,\varphi}:t\in [0,T]\mapsto\int_{\R^d}\beta(\rho(t,x))\varphi(x)\de x,
$$
$$
f_{\beta,\varphi}^\nu:t\in [0,T]\mapsto\int_{\R^d}\beta(\rho^\nu(t,x))\varphi(x)\de x.
$$
If we compute the time derivative we get
\begin{equation}\label{eq:lem56}
\dot{f}_{\beta,\varphi}=\int_{\R^2}\beta(\rho(t,x))b(t,x)\cdot\nabla\varphi(x)\de x,
\end{equation}
\begin{equation}\label{eq:lem57}
\begin{aligned}
\dot{f}_{\beta,\varphi}^\nu&=\int_{\R^2}\beta(\rho^\nu(t,x))b^{\nu}(t,x)\cdot\nabla\varphi(x)\de x+\nu\int_{\R^d}\beta(\rho^\nu(t,x))\Delta\varphi(x)\de x \\
&-\nu\int_{\R^d}|\nabla\rho^\nu(t,x)|^2\beta''(\rho^\nu(t,x))\varphi(x)\de x.
\end{aligned}
\end{equation}
Since $\beta$ is a bounded function and $\rho^\nu$ converges a.e. to $\rho$, by dominated convergence we readily conclude that for any $k<\infty$
\begin{equation}\label{eq:convbeta}
\beta(\rho^\nu)\to \beta(\rho)\hspace{0.5cm}\mbox{ in }L^{k}_{\mathrm{loc}}((0,T)\times\R^d).
\end{equation}
We write the equation for $\beta(\rho^\nu)$
\begin{align}
&\int_0^T\int_{\R^d}\beta(\rho^\nu)\left(\partial_t\varphi+b^{\nu}\cdot\nabla\varphi\right)\de x\de t+\int_{\R^d}\beta(\rho_0^\nu)\varphi_{|_{t=0}}\de x \nonumber
\\&=\nu\int_0^T\int_{\R^d}\beta(\rho^\nu)\Delta\varphi \de x-\nu\int_0^T\int_{\R^d}|\nabla\rho^\nu|^2\beta''(\rho^\nu)\varphi\de x\de t,\label{eq:termine_rinormalizzazione}
\end{align}
and by letting $\nu\to 0$ and using \eqref{eq:convbeta} and that $\varphi$ has compact support, since we know that $\rho$ is a renormalized solution of \eqref{eq:TE}, the right hand side must vanish.\\
Then, looking at \eqref{eq:lem56} and \eqref{eq:lem57}, we get that $\dot{f}^\nu_{\beta,\varphi}$ converges in $L^1((0,T))$ to $\dot{f}_{\beta,\varphi}$ which eventually implies that
$$
\int_{\R^d}\beta(\rho^\nu(t,x))\varphi(x)\de x\to\int_{\R^d}\beta(\rho(t,x))\varphi(x)\de x\hspace{0.5cm}\mbox{ uniformly in }[0,T].
$$
By approximation we can take $\beta(s)=s^q$ and $\varphi=\chi_R$, the indicator of the ball of radius $R>0$, and finally we get that
$$
\|\rho^\nu(t)\|_{L^q(B_R)}\to\|\rho(t)\|_{L^q(B_R)} \hspace{0.5cm}\mbox{ uniformly in }[0,T].
$$
\\
\underline{Step 4}\hspace{0.5cm}\emph{Convergence in $C([0,T];L^q_{\mathrm{loc}}(\R^d))$, $1\leq q<\infty$.}\\
\\
By Step 2 we have that for any $t\in[0,T]$ and any $\{t_{\nu}\}_{\nu}\subset [0,T]$ such that $t_{\nu}\to t$ 
\begin{equation}\label{eq:cc1}
\int_{\R^2}\rho^\nu(t_\nu,x)\varphi(x)\de x\to \int_{\R^2}\rho(t,x)\varphi(x)\de x,
\end{equation}
while by Step 3 we get that
\begin{equation}\label{eq:cc2}
\int_{B_R}|\rho^\nu(t_\nu,x)|^q\de x\to \int_{B_R}|\rho(t,x)|^q\de x,\quad \mbox{ for any } R>0.
\end{equation}
From \eqref{eq:cc1} and \eqref{eq:cc2} we easily infer that for $1<q<\infty$
\begin{equation}\label{eq:convlocal}
\rho^\nu\to\rho\quad \mbox{ in } C([0,T];L^q_{\mathrm{loc}}(\R^d)).
\end{equation}
Since the convergence is local in space we deduce that \eqref{eq:convlocal} holds also in the case $q=1$.\\
\\
\underline{Step 5}\hspace{0.5cm}\emph{Convergence in $C([0,T];L^q(\R^d))$, $1\leq q<\infty$.}\\
\\
Let $r>0$, then
\begin{equation}
\|\rho^\nu(t,\cdot)-\rho(t,\cdot)\|_{L^q}^q\leq \int_{B_r}|\rho^\nu(t,x)-\rho(t,x)|^q\de x+\int_{B^c_r}|\rho^\nu(t,x)|^q\de x+\int_{B^c_r}|\rho(t,x)|^q\de x.
\end{equation}
By the previous step we know that the first term on the right hand side converges to $0$ as $\nu\to 0$ for any fixed $r>0$. The remaining two terms can be made arbitrary small independently from $\nu$ if we prove that for $1\leq q<\infty$ it holds that for any $\eta>0$ there exists $r>0$, independent from $\nu$ such that
\begin{equation}\label{eq:tightness}
\sup_{t\in(0,T)}\left(\int_{B_r^c}|\rho^\nu(t,x)|^q\de x+\int_{B_r^c}|\rho(t,x)|^q\de x\right)<\eta.
\end{equation}
The following argument holds for $\nu\geq 0$. Let $r,R>0$ such that $2r<R$ and let us consider a positive test function $\psi_r^R$ for which
\begin{equation}
\psi_r^R(x)=
\begin{cases}
0 \hspace{1cm}\mbox{if } 0<|x|<r,\\
1 \hspace{1cm}\mbox{if } 2r<|x|<R,\\
0 \hspace{1cm}\mbox{if } |x|>2R,
\end{cases}
\end{equation}
such that $0\leq \psi_r^R\leq 1$ and 
\begin{equation}\label{boundpsi}
|\nabla \psi_r^R|\leq \frac{C}{r}, \hspace{0.5cm}|\nabla^2 \psi_r^R|\leq \frac{C}{r^2}.
\end{equation}
Let $t\in(0,T)$ and $\beta(s)=s^q$. Multiply the equation \eqref{eq:TEe} by $\beta'(|\rho^\nu|)\psi_r^R$ and integrate in space and in time. We get that
\begin{align*}
\int_{\R^d}\beta(|\rho^\nu(t)|)\psi_r^R \de x & \leq \int_{\R^d}\beta(|\rho_0^\nu|)\psi_r^R\de x + \int_0^t\int_{\R^d}\beta(|\rho^\nu|)|b^{\nu}||\nabla \psi_r^R|\de x \de t\\
&+\nu \int_0^t\int_{\R^d}\beta(|\rho^\nu|)|\Delta \psi_r^R|\de x \de t.
\end{align*}
By using \eqref{est:rhonu} and \eqref{eq:h2} in the case $\nu>0$, the analogous bounds for $\rho$ and $b$ in the case $\nu=0$, and \eqref{boundpsi}, after sending $R\to\infty$ we obtain that
\begin{align*}
\int_{B_r^c}|\rho^\nu(t,x)|^q\de x &\leq \int_{B_r^c}|\rho_0^\nu(x)|^q\de x +\frac{C}{r}\|\rho^\nu\|_{L^\infty(L^\infty)}\int_0^T\int_{\R^d}|b^{\nu}_1(t,x)|\de x \de t \\
&+\frac{C}{r} \|b^{\nu}_2(t,\cdot)\|_{L^\infty(L^\infty)}\int_0^T\int_{\R^d}|\rho^\nu(t,x)|^q\de x \de t \\
&+\frac{C\nu}{r^2} \int_0^T\int_{\R^d}|\rho^\nu(t,x)|^q\de x \de t\\
&\leq \int_{B_r^c}|\rho_0^\nu(x)|^q\de x +\frac{C}{r}+\frac{C}{r^2},
\end{align*}
where the constant $C$ is independent on $\nu$ and $t$. Next, note that by \eqref{eq:convrho0}, we have that $\rho^{\nu}_0\to\rho_0$ strongly in $L^{q}(\R^d)$ and therefore given $\eta>0$ there exists $r>0$ such that 
\begin{equation*}
\int_{B_r^c}|\rho_0^\nu(x)|^q\de x\leq \frac{\eta}{2}, 
\end{equation*}
and the same holds for $\rho_0$. Finally, choosing $r$ such that also $\frac{C}{r}+\frac{C}{r^2}\leq \frac{\eta}{2}$ we deduce \eqref{eq:tightness}. 
\end{proof}

\subsection{Proof of Theorem \ref{teo:strong_convergence_vort}}

\begin{proof}
We divide the proof in several steps.\\
\\
\underline{Step 1}\hspace{0.5cm}{\em Weak convergence of the vorticity.}\\
\\
As in  Step 1 of Theorem \ref{teo:main2} we have that 
\begin{equation}\label{eq:weakconvomr2}
\omv\weaktos\om\quad\mbox{ weakly* in }L^{\infty}((0,T);L^{p}(\R^2)). 
\end{equation}
Indeed, the same proof holds also in the case $p=1$ provided we show that $\{\omv\}_{\nu}$ is equi-integrable in $L^{1}((0,T)\times\R^2)$. To prove this we start by noticing that, since $ \omv_{0}\rightarrow \omega_{0}\textrm{ in }L^{1}(\R^2)$, for any $\e>0$ there exist $C_\e$, 
$\omvepi$, and $\omveii$  such that 	
\begin{equation}\label{eq:3}
\begin{aligned}
&\omv_{0}=\omvepi+\omveii,
&\|\omvepi\|_{1}\leq \e,\hspace{0.4cm}\mbox{ and }\hspace{0.4cm}\|\omveii\|_{\infty}\leq C_\e.
\end{aligned}
\end{equation}
We also have that both $\omvepi$ and $\omveii$ are in $L^{1}(\R^2)\cap L^{\infty}(\R^{2})$ with bounds  depending on $\nu$ and $\e$.\\
Let us consider the unique weak solution $\omvep\in L^{\infty}((0,T);L^{1}(\R^2)\cap L^{\infty}(\R^{2}))$ of the linear problem 
\begin{equation}\label{eq:5}
\begin{cases}
\partial_t \omvep-\nu\Delta \omvep+\uv\cdot\nabla \omvep=0,\\
\omvep(0,x)=\omvepi.
\end{cases}
\end{equation}
By standard $L^p$-estimates we have that 
\begin{equation}\label{eq:7}
\|\omvep(t)\|_{1}\leq \|\omvepi\|_{1}\leq \e.
\end{equation}
\\
Next, we consider the unique weak solution $\omvei\in  L^{\infty}((0,T);L^{1}(\R^2)\cap L^{\infty}(\R^{2}))$ of the linear problem 
\begin{equation}\label{eq:8}
\begin{cases}
\partial_t \omvei-\nu\Delta \omvei+\uv\cdot\nabla \omvei=0,\\
\omvei(0,x)=\omveii.
\end{cases}
\end{equation}
By the maximum principle we have that  
\begin{equation}\label{eq:10}
\|\omvei(t)\|_{\infty}\leq \|\omveii\|_{\infty}\leq C_{\e},
\end{equation}
where $C_{\e}$ is the same constant as in \eqref{eq:3}. Moreover, for $C$ independent on $\nu$ and $\e$ we also have \begin{equation}\label{eq:11}
\|\omvei(t)\|_{1}\leq \|\omveii\|_{1}\leq C,
\end{equation}
where the last inequality in \eqref{eq:11} follows from \eqref{eq:3}. 
Next, we want to prove that $\omvei$ is small at infinity. 
Let $r$ and $R$ be such that $\tilde{R}<r<R /2$ and let $\psi_{r}^{R}\in C^{\infty}_{c}(\R^{2})$ be the cut-off function defined in Lemma \ref{teo:strong_convergence_TE}. Then, since $\omvei$ satisfies
\begin{equation*}
\partial_t|\omvei|+\uv\cdot\nabla|\omvei|-\nu\Delta|\omvei|\leq 0, 
\end{equation*}
and $\psi_{r}^{R}$ is positive we can easily deduce that 
\begin{equation}\label{eq:equi9}
\int|\omvei|\psi_{r}^{R}\,\de x\leq \iint |\uv||\omvei||\nabla\psi_{r}^{R}|\,\de x\de t+\nu\iint |\omvei||\Delta\psi_{r}^{R}|\,\de x\de t,
\end{equation}
and after sending $R\to\infty$  we have 
\begin{equation}\label{eq:equi10}
\begin{aligned}
\int_{B_{2r}^c}|\omvei|\de x & \leq\frac{1}{r}\iint|\uv||\omvei|\,\de x\de t   +\frac{\nu}{r^2}\iint|\omvei|\,\de x\de t. 
\end{aligned}
\end{equation}
Let us now decompose the Biot-Savart kernel $K= K_1+K_2$, where $K_{1}=K\chi_{B_{1}(0)}\in L^{1}(\R^2)$ and $K_{2}=K\chi_{B_{1}(0)^{c}}\in L^{\infty}(\R^2)$. The decomposition of the kernel induces the decomposition $\uv=\uv_{1}+\uv_{2}$  and, by Young's inequality (for the convolution), we have  that $\{\uv_1\}_{\nu}$ is bounded in $ L^{\infty}((0,T);L^{1}(\R^2))$ and $\{\uv_2\}_{\nu}$ is bounded in $L^{\infty}((0,T)\times\R^2)$ and therefore from \eqref{eq:equi10} for some $C$ independent from $\nu$ and $\e$ we get that for a.e. $t\in (0,T)$
\begin{equation}\label{eq:equi11}
\begin{aligned}
\int_{B_{2r}^c}|\omvei|\,\de x & \leq\frac{C(C_{\e}+1)}{r}   +\frac{C}{r^2},
\end{aligned}
\end{equation}
which implies the existence of $r_{\e}$ such that for a.e. $t\in (0,T)$
\begin{equation}\label{eq:211}
\int_{B_{r_{\e}}^c}|\omvei|\,\de x\leq \e.
\end{equation}
Next, we notice that for fixed $\nu$ we have that 
$\omv\in L^{\infty}((0,T);L^{1}(\R^2)\cap L^{\infty}(\R^2))$, and $\omv$ solves 
\begin{equation}\label{eq:w}
\begin{cases}
\partial_t\omv+\uv\cdot\nabla\omv = \nu\Delta\omv ,\\
\omv|_{t=0}=\omv_0.       
\end{cases}
\end{equation}
Then, fix $\e>0$ and define $\hat{\omega}_{\nu,\e}:=\omvep+\omvei$. Then, $\hat{\omega}_{\nu,\e}\in L^{\infty}((0,T);L^{1}(\R^2) \cap L^{\infty}(\R^2))$, and $\hat{\omega}_{\nu,\e}$ solves 
\begin{equation}\label{eq:w1}
\begin{cases}
\partial_t\hat{\omega}_{\nu,\e}+\uv\cdot\nabla\hat{\omega}_{\nu,\e} = \nu\Delta\hat{\omega}_{\nu,\e} ,\\
\hat{\omega}_{\nu,\e}|_{t=0}=\omv_0.       
\end{cases}
\end{equation}
Then the uniqueness of the linear problem implies that $$\omv=\hat{\omega}_{\nu,\e}=\omvep+\omvei.$$
In conclusion we have proved that for any $\e>0$ there exist $C_\e$, $r_\e$, $\omvep$ and $\omvei$ such that for a.e. $t\in(0,T)$
\begin{equation*}
\begin{aligned}
&\omv=\omvep+\omvei,
&\|\omvep(t)\|_{1}\leq \e,\\
&\|\omvei(t)\|_{\infty}\leq C_{\e},
&\int_{B^{c}_{r_{\e}}}\omvei(t,x)\,\de x\leq \e. 
\end{aligned}
\end{equation*}
By integrating in time, since $T$ is assumed finite, we easily get that $\{\omv\}_{\nu}$ is equi-integrable in $L^{1}((0,T)\times\R^{2})$.\\
\\
\underline{Step 2}\hspace{0.5cm}{\em Strong convergence of the velocity.}\\
\\
We first recall that for any $p\geq 1$, the kernel $K:L^{p}(\R^2)\to L^q_{\mathrm{loc}}(\R^2)$ is a compact operator, when $q$ is such that
\begin{equation}\label{qp}
1+\frac{1}{q}-\frac{1}{p}> \frac{1}{2}.
\end{equation}
\\
Moreover, it is a classical fact (see \cite{DPM})  that, for some $s,m>0$, we also have that 
$$
\{u^\nu\}\mbox{ is bounded in } \mathrm{Lip}([0,T];W^{-s,m}_{\mathrm{loc}}(\R^2)).
$$
Then, we easily deduce that for $p>1$ we can upgrade the convergence \eqref{eq:conv} to
$$
u^\nu\to u \hspace{0.5cm} \mbox{in } L^2((0,T);L^2_{\mathrm{loc}}(\R^2)),
$$
while for $p=1$ we have
$$
u^\nu\to u \hspace{0.5cm} \mbox{in } L^q((0,T);L^q_{\mathrm{loc}}(\R^2)),
$$
for any $1\leq q<2$.
\\
\\
\underline{Step 3}\hspace{0.5cm}{\em Strong convergence of the vorticity.}\\
\\
The proof is based on an $\e$-third argument as in \cite{CDE}. Let $\psi_n$ be a standard mollifier on $\R^2$, we introduce the following linear problems
\begin{equation}\label{eq:lin1}
\begin{cases}
\partial_t\omega^\nu_n+u^\nu\cdot\nabla\omega^\nu_n=\nu\Delta\omega^\nu_n,\\
\omega^\nu_n(0,\cdot)=\omega_0^\nu*\psi_n,
\end{cases}
\end{equation}
and
\begin{equation}\label{eq:lin2}
\begin{cases}
\partial_t\omega_n+u\cdot\nabla\omega_n=0,\\
\omega_n(0,\cdot)=\omega_0*\psi_n.
\end{cases}
\end{equation}
Note that the Cauchy problems \eqref{eq:lin1}, \eqref{eq:lin2} are linear since the vector fields $u^\nu$ and $u$ are fixed and do not depend on the solution itself contrary to what happens for the Euler and the Navier-Stokes equations. Moreover, there exists a unique smooth solution $\omega^\nu_n$ of $\eqref{eq:lin1}$ because $u^\nu$ is smooth, and a unique solution $\omega_n\in L^\infty((0,T);L^1(\R^2)\cap L^\infty(\R^2))$ as a consequence of the uniqueness theorems in \cite{DPL} for $p>1$ and \cite{CNSS} for $p=1$.\\

By triangular inequality we have that
\begin{equation}\label{eq:splittingeu}
\begin{aligned}
\sup_{t\in[0,T]}\|\omega^\nu(t)-\omega(t)\|_{L^p}\leq & \underbrace{\sup_{t\in[0,T]}\|\omega^\nu(t)-\omega^\nu_n(t)\|_{L^p}}_{(I)}
\\ &+\underbrace{\sup_{t\in[0,T]}\|\omega^\nu_n(t)-\omega_n(t)\|_{L^p}}_{(II)}
\\ &+\underbrace{\sup_{t\in[0,T]}\|\omega_n(t)-\omega(t)\|_{L^p}}_{(III)}.
\end{aligned}
\end{equation}
We estimate separately the three terms on the right hand side of \eqref{eq:splittingeu}.
Regarding $(I)$, we notice that the difference $\omega^\nu-\omega^\nu_n$ satisfies the equation
\begin{equation}
\partial_t\left(\omega^\nu-\omega^\nu_n\right)+u^\nu\cdot\nabla\left(\omega^\nu-\omega^\nu_n\right)=\nu\Delta\left(\omega^\nu-\omega^\nu_n\right).
\end{equation}
Therefore, we easily get for any $t\in(0,T)$ that 
\begin{align*}
\|\omega^\nu(t)-\omega^\nu_n(t)\|_{L^p}&\leq \|\omega^\nu_0-\omega^\nu_{0,n}\|_{L^p}
\end{align*}
which is small for $n$ large enough independently from $\nu$.\\
\\
Next, we consider $(III)$: since $\omega$ is a renormalized solution, due to the uniqueness of the linear problem (see \cite{DPL} and \cite{CNSS}) $\om$ is also Lagrangian and therefore 
$$
\omega(t,x)=\omega_0(X_{t,0}(x)),
$$
where $X$ is the unique regular Lagrangian flow of $u$. Moreover, the unique solution $\omega_n$ of \eqref{eq:lin2} is also renormalized and then Lagrangian and therefore is given by
$$
\omega_n(t,x)=\omega_{0,n}(X_{t,0}(x)).
$$
By using that $X$ is measure-preserving
\begin{align*}
\sup_{t\in[0,T]}\|\omega_n(t,\cdot)-\omega(t,\cdot)\|_{L^p}^p=&\sup_{t\in[0,T]}\int_{\R^2}\left| \omega_{0,n}(X_{t,0}(x))- \omega_0(X_{t,0}(x))\right|^p\de x\\ =&\int_{\R^2}\left| \omega_{0,n}(y)-\omega_0(y)\right|^p\de y=\|\omega_{0,n}-\omega_0\|_{L^p}^p,
\end{align*}
which goes to $0$ as $n\to\infty$.\\
\\
Finally, we consider the term $(II)$ and we note that for fixed $n$ the sequence of solutions $\{\omega_{n}^{\nu}\}_{\nu}$, the sequence of velocity fields $\{\uv\}_{\nu}$, the limit solution $\omega_n$, and the limit vector field $u$ satisfy the hypothesis of Lemma \ref{teo:strong_convergence_TE}. Therefore, for fixed $n$ the term $(II)$ goes to zero as $\nu\to0$ and the proof is concluded. 
\end{proof}

\section{Conservation of the energy}
In this last section we prove that solutions of the $2D$ Euler equations obtained in the vanishing viscosity limit conserve the energy. In particular, we extend the result in \cite{CFLS} to the case when the Euler equations are considered on the whole space $\R^2$. The strategy we adopt is similar to the one we used in \cite{CCS3} and it is combined with the results of \cite{CFLS}. We start by introducing some additional notation. We denote with $\star$ the following variant of the convolution:
\begin{align}
v\star w&=\sum_{i=1}^2 v_i*w_i \hspace{1cm}\mbox{if }v,w\mbox{ are vector fields in }\R^2,\label{4.1}\\
A\star B&=\sum_{i,j=1}^2A_{ij}*B_{ij} \hspace{0.5cm}\mbox{if }A,B\mbox{ are matrix-valued functions in }\R^2. \label{4.2}
\end{align}
With the notations above it is easy to check that, if $f:\R^2\to\R$ is a scalar function and $v:\R^2\to\R^2$ is a vector field, then
$$
f*\curl v=\nabla^\perp f\star v,
$$
$$
\nabla^\perp f\star \dive(v\otimes v)=\nabla\nabla^\perp f\star(v\otimes v),
$$
$$
v_i*\Delta f= \Delta v_i*f.
$$
A peculiar fact of the two-dimensional Euler equations is that the velocity field is in general not globally square integrable: this is due to the fact that the Biot-Savart kernel fails to be square integrable at infinity.  To have a well-defined kinetic energy we need to require that the vorticity has zero mean value. In fact, the following proposition holds true, see \cite{MB}.
\begin{prop}\label{prop:dpm}
An incompressible velocity field in $\R^2$ with vorticity of compact support has finite kinetic energy if and only if the vorticity has zero mean value, that is
\begin{equation}\label{zeromean}
\int_{\R^2}|u(t,x)|^2\de x<\infty\iff\int_{\R^2}\omega(t,x)\de x=0.
\end{equation}
\end{prop}
The main result of this section is the following. We stress that the proof below does not hold in the case $p=1$ since the convergence \eqref{conv:glob} fails in this case, as already pointed out in Step 2 in the proof of Theorem \ref{teo:strong_convergence_vort}.
\begin{thm}\label{teo:energy}
Let $p\in(1,\infty)$ and $\omega_0\in L^p_c(\R^2)$ verifying \eqref{zeromean}. Let $u^\nu, u$ as in Theorem \ref{teo:strong_convergence_vort}. Then, $u^\nu$ satisfies the following convergence
\begin{equation}\label{conv:glob}
u^\nu\to u \hspace{1cm}\mbox{in }C([0,T];L^2(\R^2)),
\end{equation}
and $u$ conserves the energy, that is
\begin{equation}\label{eq:consen}
\|u(t)\|_{L^2}=\|u_0\|_{L^2}, \hspace{0.5cm}\forall t\in[0,T].
\end{equation}
\end{thm}
\begin{proof}
We recall that the parameter $\nu$ is always supposed to vary over a countable set, therefore given the sequence $\nu_n\to 0$, we denote with $u^n$ and $\omega^n$ the sequences $u^{\nu_n}$ and $\omega^{\nu_n}$. We divide the proof in several steps.\\
\\
\underline{Step 1}\hspace{0.5cm}\emph{A Serfati identity with fixed vorticity.}\\
\\
In this step we derive a formula for the approximate velocity $u^n$.\\
Let $a\in C^\infty_c(\R^2)$ be a smooth function such that $a(x)=1$ if $|x|<1$ and $a(x)=0$ for $|x|>2$. Differentiating in time the Biot-Savart formula we obtain that for $i=1,2$
\begin{align}
\partial_s u^n_i(s,x)&=K_i*(\partial_s \omega^n)(s,x) \nonumber\\
&=(aK_i)*(\partial_s \omega^n)(s,x)+[(1-a)K_i]*(\partial_s \omega^n)(s,x).\label{proof:convglob1}
\end{align}
Now we use the equation \eqref{eq:vort_NS} for $\omega^n$ obtaining
$$
\partial_s \omega^n=-u^n\cdot\nabla \omega^n+\nu_n\Delta \omega^n,
$$
and substituting in \eqref{proof:convglob1} we obtain
\begin{equation}
\partial_s u^n_i=(aK_i)*(\partial_s \omega^n)-[(1-a)K_i]*(v^n\cdot\nabla \omega^n)+[(1-a)K_i]*\left(\nu_n\Delta \omega^n\right).
\end{equation}
By the identity
$$
u^n\cdot\nabla \omega^n=\curl(u^n\cdot\nabla u^n)=\curl\dive(u^n\otimes u^n)
$$
we obtain that
\begin{equation}
[(1-a)K_i]*(u^n\cdot\nabla \omega^n)=\left(\nabla\nabla^\perp[(1-a)K_i]\right)\star(u^n\otimes u^n),\label{proof:convglob2}
\end{equation}
while by the properties of the convolution
\begin{equation}
[(1-a)K_i]*\left(\nu_n\Delta \omega^n\right)=\left(\Delta[(1-a)K_i]\right)*\left( \nu_n\omega^n\right), \label{proof:convglob3}
\end{equation}
where the notation $\star$ was introduced in \eqref{4.1} and \eqref{4.2}.
Substituting the expressions \eqref{proof:convglob2} and \eqref{proof:convglob3} in \eqref{proof:convglob1} and integrating in time we have that $u^n$ satisfies the following formula:
\begin{equation}\label{eq:serfativv}
\begin{split}
u^n_i(t,x)&=u^n_i(0,x)+(aK_i)*\left(\omega^n(t,\cdot)-\omega^n(0,\cdot)\right)(x)\\
&-\int_0^t\left(\nabla\nabla^\perp[(1-a)K_i]\right)\star(u^n(s,\cdot)\otimes u^n(s,\cdot))(x)\de s\\ &+\int_0^t\left(\Delta[(1-a)K_i]\right)*\left( \nu_n\omega^n\right)\de s.
\end{split}
\end{equation}
\\
\underline{Step 2}\hspace{0.5cm}\emph{$u^n$ is a Cauchy sequence in $C([0,T];L^2(\R^2))$.}\\
\\
Using formula \eqref{eq:serfativv} we can prove that $u^n$ is a Cauchy sequence. We consider $u^n,u^m$ with $n,m \in\N$. By linearity of the convolution we have that $u^n-u^m$ satisfies the following
\begin{equation}\label{diff}
\begin{split}
u^n_i(t,x)&-u^m_i(t,x)=\underbrace{u^n_i(0,x)-u^m_i(0,x)}_{(I)}\\ &
+\underbrace{(aK_i)*(\omega^n(t,\cdot)-\omega^m(t,\cdot))(x)}_{(II)}+\underbrace{(aK_i)*(\omega^m_0-\omega^n_0)(x)}_{(III)}\\ 
&-\int_0^t\underbrace{\left(\nabla\nabla^\perp[(1-a)K_i]\right)\star\left[u^n(s,\cdot)\otimes u^n(s,\cdot)-u^m(s,\cdot)\otimes u^m(s,\cdot)\right](x)}_{(IV)}\de s\\ &+\int_0^t\underbrace{\left(\Delta[(1-a)K_i]\right)*\left( \nu_n\omega^n(s,\cdot)-\nu_m\omega^m(s,\cdot)\right)}_{(V)}\de s.
\end{split}
\end{equation}
In order to estimate $\|u^n(t)-u^m(t)\|_{L^2}$ we estimate separately the $L^2$ norms of the terms on the right hand side of \eqref{diff}. We start by estimating $(I)$: given $\eta>0$, since the initial datum $u^n_0$ converges in $L^2$ to $u_0$, we have that there exists $N_1$ such that 
\begin{equation}\label{I}
\|u^n_0-u^m_0\|_{L^2}<\eta \hspace{1cm}\mbox{for any }n,m>N_1.
\end{equation}
We deal now with $(II),(III)$: if $\omega_0\in L^p_c(\R^2)$ with $1<p<2$, by Young's convolution inequality we have that
\begin{equation}
\label{young1}
\|(aK)*(\omega^n(t)-\omega^m(t))\|_{L^2}\leq \|aK\|_{L^q}\|\omega^n(t)-\omega^m(t)\|_{L^p},
\end{equation}
where $1<q<2$ is such that $1+\frac{1}{2}=\frac{1}{q}+\frac{1}{p}$, while for $p\geq 2$
\begin{equation}
\|(aK)*(\omega^n(t)-\omega^m(t))\|_{L^2}\leq \|aK\|_{L^1}\|\omega^n(t)-\omega^m(t)\|_{L^2}. 
\label{young2}
\end{equation}
Since $\|aK\|_{L^q}\leq \|K\|_{L^q(B_2)}$ and $K\in L^q_{\mathrm{loc}}(\R^2)$ for any $1\leq q<2$, by the strong convergence of $\omega^n$ proved in Theorem \ref{teo:strong_convergence_vort}, there exists $N_2$ such that
\begin{equation}\label{II+III}
\|(aK)*(\omega^n(t)-\omega^m(t))\|_{L^2}+\|(aK)*(\omega^n_0-\omega^m_0)\|_{L^2}<C\eta,
\end{equation}
for any $n,m>N_2$. We deal now with $(IV)$: by Young's convolution inequality we have that
\begin{align}
\|\nabla\nabla^\perp&[(1-a)K]\star(u^n(s)\otimes u^n(s)-u^m(s)\otimes u^m(s))\|_{L^2}\nonumber\\ &\leq \|\nabla\nabla^\perp[(1-a)K]\|_{L^2}\underbrace{\|u^n(s)\otimes u^n(s)-u^m(s)\otimes u^m(s)\|_{L^1}}_{(IV*)}\label{arr}.
\end{align}
We add and subtract $u^n(s,\cdot)\otimes u^m(s,\cdot)$ in $(IV*)$ and by H\"older inequality we have
\begin{align*}
\|u^n(s)\otimes u^n(s)&-u^m(s)\otimes u^m(s)\|_{L^1}\\ &\leq \left(\|u^n(s)\|_{L^2}+\|u^m(s)\|_{L^2}\right)\|u^n(s)-u^m(s)\|_{L^2}.
\end{align*}
For the first factor in \eqref{arr} we have that
$$
\nabla\nabla^\perp[(1-a)K_i]=-(\nabla\nabla^\perp a )K_i-\nabla^\perp a\nabla K_i-\nabla a \nabla^\perp K_i+(1-a)\nabla\nabla^\perp K_i,
$$
and it is easy to see that each term on the right hand side has uniformly bounded $L^2$-norm. Then we have that
\begin{equation}\label{IV}
\begin{split}
\int_0^t\|\nabla\nabla^\perp[(1-a)K]\star(u^n(s)&\otimes u^n(s)-u^m(s)\otimes u^m(s)\|_{L^2}\de s\\ & \leq C \|u_0\|_{L^2}\int_0^t\|u^n(s)-u^m(s)\|_{L^2}\de s.
\end{split}
\end{equation}
Finally, we deal with $(V)$: again by Young's inequality we have that
\begin{align*}
&\left\|\left(\Delta[(1-a)K_i]\right)*\left(\nu_n\omega^n(s)-\nu_m\omega^m(s)\right)\right\|_{L^2}\\
&\leq \nu_n\|\Delta[(1-a)K_i]\|_{L^q}\|\omega^n(s)-\omega^m(s)\|_{L^p}+\left| \nu_m-\nu_n \right|\|\Delta[(1-a)K_i]\|_{L^q}\|\omega^m(s)\|_{L^p},
\end{align*}
where $p$ and $q$ are chosen as in \eqref{young1} or \eqref{young2} depending on whether $p$ is bigger or smaller than $2$. Since $\Delta K_i$ is in $L^q(B_1^c)$, a straightforward computation shows that $\Delta[(1-a)K]$ is bounded in $L^q$. So there exists $N_3$ such that for all $n,m>N_3$ we have that
\begin{equation}\label{V}
\left\|\left(\Delta[(1-a)K]\right)*\left( \nu_n\omega^n(s)-\nu_m\omega^m(s)\right)\right\|_{L^2}\leq C\eta.
\end{equation}
Then, putting together \eqref{I},\eqref{II+III},\eqref{IV} and \eqref{V} we obtain that for all $n,m>N:=\max\{N_1,N_2,N_3\}$
\begin{equation}
\|u^n(t)-u^m(t)\|_{L^2}\leq C\left(\eta+\int_0^t\|u^n(s)-u^m(s)\|_{L^2}\de s\right),
\end{equation}
and by Gr\"onwall's lemma
\begin{equation}\label{proof:cauchyfinal}
\|u^n(t)-u^m(t)\|_{L^2}\leq C\eta.
\end{equation}
Taking the supremum in time in \eqref{proof:cauchyfinal} we obtain \eqref{conv:glob}. \\
\\
\underline{Step 3}\hspace{0.5cm}\emph{Conservation of energy.}\\
\\
First of all, we can restrict our attention to the case $\omega_0\in L^p_c(\R^2)$ with $1<p<3/2$, otherwise there is nothing to prove (see \cite{CFLS}). Let $u^\nu$ be the unique smooth solution of the Navier-Stokes equations \eqref{eq:ns} and let $\omega^\nu=\curl u^\nu$, which satisfies the equation
\begin{equation}\label{eq:NSvort}
\partial_t \omega^\nu+u^\nu\cdot\nabla\omega^\nu=\nu\Delta\omega^\nu.
\end{equation}
Multiplying \eqref{eq:NSvort} by $\omega^\nu$ and integrating over $\R^2$ we obtain
\begin{equation}
\label{eq:bilancio_vort}
\frac{\de}{\de t}\|\omega^\nu(t)\|_{L^2}^2=-2\nu\|\nabla \omega^\nu(t)\|_{L^2}^2.
\end{equation}
By using the Gagliardo-Niremberg inequality we have that
\begin{equation}\label{proof:consvanGN}
\|\omega^\nu(t)\|_{L^2}\leq \|\nabla\omega^\nu(t) \|_{L^2}^{1-\frac{p}{2}}\|\omega^\nu(t) \|_{L^p}^{\frac{p}{2}},
\end{equation}
from which it follows that
\begin{equation}\label{proof:consvaninvGN}
-2\nu\|\nabla\omega^\nu(t) \|_{L^2}\leq -2\nu\|\omega^\nu(t)\|_{L^2}^{\frac{4}{2-p}}\|\omega^\nu(t) \|_{L^p}^{-\frac{2p}{2-p}}.
\end{equation}
We multiply \eqref{eq:NSvort} by $|\omega^\nu|^{p-2}\omega^\nu$ and integrating on $\R^2$ we also get
$$
\|\omega^\nu(t)\|_{L^p}\leq\|\omega_0^\nu\|_{L^p},
$$
and substituting in \eqref{proof:consvaninvGN} and in \eqref{eq:bilancio_vort} we obtain
\begin{equation}\label{proof:convanineq}
\frac{\de}{\de t}\|\omega^\nu(t)\|_{L^2}^2\leq -2\nu\|\omega^\nu(t)\|_{L^2}^{\frac{4}{2-p}}\|\omega_0^\nu\|_{L^p}^{-\frac{2p}{2-p}}.
\end{equation}
Define $y(t)=\|\omega^\nu(t)\|_{L^2}^2$ and take $C_0$ such that $\|\omega_0^\nu\|_{L^p}^{-\frac{2p}{2-p}}\leq C_0$, where we can assume that $C_0$ is independent from $\nu$ because of the (strong) convergence of $\omega_0^\nu$ towards $\omega_0$ in $L^p$. Then, integrating in time in \eqref{proof:convanineq} we obtain
$$
y(t)^{-\frac{p}{2-p}}-y(0)^{-\frac{p}{2-p}}\geq \frac{2\nu p C_0}{2-p}t,
$$
from which it follows that
\begin{equation}\label{proof:consvanmain}
\|\omega^\nu(t)\|_{L^2}^2\leq\left(\|\omega_0^\nu\|_{L^2}^{-\frac{2p}{2-p}}+\frac{2\nu p C_0 t}{2-p}\right)^{-\frac{2-p}{p}}.
\end{equation}
Smooth solutions of the 2D Navier-Stokes equations satisfy the energy identity
$$
\frac{\de}{\de t}\|u^\nu(t)\|_{L^2}^2=-2\nu\|\nabla u^\nu(t)\|_{L^2}^2,
$$
and rewriting the right hand side in terms of the vorticity we have
\begin{equation}\label{proof:consvanenerbal}
\frac{\de}{\de t}\|u^\nu(t)\|_{L^2}^2=-2\nu\|\omega^\nu(t)\|_{L^2}^2.
\end{equation}
Hence, integrating in time in \eqref{proof:consvanenerbal} and using \eqref{proof:consvanmain} we deduce that
\begin{align}
0&\geq \|u^\nu(t)\|_{L^2}^2-\|u^\nu_0\|_{L^2}^2 \geq -2\nu\int_0^t\left(\|\omega_0^\nu\|_{L^2}^{-\frac{2p}{2-p}}+\frac{2\nu p C_0 t}{2-p}\right)^{-\frac{2-p}{p}}\de s\nonumber\\ 
\label{proof:convvanfinal}
& =-\frac{2-p}{2C_0(p-1)}\left[\left(\|\omega_0^\nu\|_{L^2}^{-\frac{2p}{2-p}}+\frac{2\nu p C_0}{2-p}t\right)^{\frac{2(p-1)}{p}} -\|\omega_0^\nu\|_{L^2}^{-\frac{2p}{2-p}} \right].
\end{align}
Now, since $\omega_0\notin L^2(\R^2)$ we must have that
$$
\lim_{\nu\to 0}\|\omega_0^\nu\|_{L^2}=+\infty,
$$
and then, being $p>1$, the right hand side of \eqref{proof:convvanfinal} vanishes as $\nu\to 0$. Therefore, by using \eqref{conv:glob} we have that
$$
0=\lim_{\nu\to 0} \left(\|u^\nu(t)\|_{L^2}^2-\|u^\nu_0\|_{L^2}^2\right)=\|u(t)\|_{L^2}^2-\|u_0\|_{L^2}^2, 
$$
which concludes the proof.
\end{proof}


\begin{thebibliography}{10}

\bibitem{AKFL}\textsc{D. M. Ambrose, J. P. Kelliher, M. C. Lopes Filho, H. J. Nussenzveig Lopes}: \emph{Serfati solutions to the 2D Euler equations on exterior domain.} J. Differential Equations, $\mathbf{259}$, (2015), 4509-4560.

\bibitem{Am} \textsc{L. Ambrosio}: \emph{Transport equation and Cauchy problem for BV vector fields.} Inventiones Mathematicae, $\mathbf{158}$, (2004), 227-260.

\bibitem{Be} \textsc{J. T. Beale}: \emph{The approximation of weak solutions to the 2-D Euler equations by vortex elements.} Multidymensional Hyperbolic Problems and Computations, Edited by J. Glimm and A. Majda, IMA Vol. Math. Appl. $\mathbf{29}$, (1991), 23-37.

\bibitem{BB1} \textsc{A. Bohun, F. Bouchut, G. Crippa}: \emph{Lagrangian flows for vector fields with anisotropic regularity.} Ann. Inst. H. Poincar\`e Anal. Non Lin\`eaire, $\mathbf{33}$, 6 (2016), 1409-1429.

\bibitem{BBC2} \textsc{A. Bohun, F. Bouchut, G. Crippa}: \emph{Lagrangian solutions to the 2D Euler system with $L^1$ vorticity and infinite energy.} Nonlinear Analysis: Theory, Methods \& Applications, $\mathbf{132}$, (2016), 160-172.

\bibitem{BJ} \textsc{D. Bresch, P.-E. Jabin}: \emph{Global existence of weak solutions for compresssible Navier-Stokes equations: Thermodynamically unstable pressure and anisotropic viscous stress tensor}. Ann. of Math., {\bf 188}, (2018) 577-684.

\bibitem{BM} \textsc{A. Bressan, R. Murray}: \emph{On Self-Similar Solutions to the Incompressible Euler Equations}. J. Differential Equations, $\mathbf{269}$, (2020), 5142-5203.

\bibitem{BS} \textsc{A. Bressan, W. Shen}: \emph{A Posteriori Error Estimates for Self-Similar Solutions to the Euler Equations}. Discrete \& Continuous Dynamical Systems - A, (2021), $\mathbf{41}$ (1), 113-130.

\bibitem{BC} \textsc{F. Bouchut, G. Crippa}: \emph{Lagrangian flows for vector fields with gradient given by a singular integral.} J. Hyperbolic Diff. Equ., $\mathbf{10}$, (2013), 235-282.

\bibitem{BN} \textsc{E. Bru\'e, Q.-H. Nguyen}: \emph{On the Sobolev space of functions with derivative of logarithmic order.} Adv. Nonlinear Anal., $\mathbf{9}$, (2020), 836-849.

\bibitem{CJ} \textsc{N. Champagnat, P. Jabin}: \emph{Strong solutions to stochastic differential equations with rough coefficients.} Ann. Probab., $\mathbf{46}$, no. 3 (2018), 1498-1541.

\bibitem{CFLS} \textsc{A. Cheskidov, M. C. Lopes Filho, H. J. Nussenzveig Lopes, R. Shvydkoy}: \emph{Energy conservation in Two-dimensional incompressible ideal fluids.} Comm. Math. Phys., $\mathbf{348}$, (2016), 129-143.

\bibitem{Ch} \textsc{J.-Y. Chemin}: \emph{A remark on the inviscid limit for two-dimensional incompressible fluids}. Commun. Part. Diff. Eq. $\mathbf{21}$, 11-12 (1996), 1771-1779.

\bibitem{CCS3}\textsc{G. Ciampa, G. Crippa, S. Spirito}: \emph{Weak solutions obtained by the vortex method for the 2D Euler equations are Lagrangian and conserve the energy.} J Nonlinear Sci $\mathbf{30}$, (2020), 2787-2820.

\bibitem{C1} \textsc{P. Constantin}: \emph{Note on loss of regularity for solutions of the 3D incompressible Euler and related equations.} Commun. Math. Phys. {\bf 104}, (1986), 311-326.

\bibitem{CI} \textsc{P. Constantin, G. Iyer}: \emph{A stochastic Lagrangian representation of the $3$-dimensional incompressible Navier-Stokes equations.} Comm. Pure Appl. Math., $\mathbf{61}$, no. 3 (2008), 330-345.

\bibitem{CDE} \textsc{P. Constantin, T. D. Drivas, T. M. Elgindi}: \emph{Inviscid limit of vorticity distributions in Yudovich class.} Comm. Pure Appl. Math., (2020), in press.

\bibitem{CW1} \textsc{P. Constantin, J. Wu}: \emph{Inviscid limit for vortex patches.} Nonlinearity, $\mathbf{8}$, 735 (1995).

\bibitem{CW2} \textsc{P. Constantin, J. Wu}: \emph{The inviscid limit for non-smooth vorticity.} Indiana Univ. Math. J., $\mathbf{45}$, no. 1 (1996): 67-81.

\bibitem{CDL} \textsc{G. Crippa, C. De Lellis}: \emph{Estimates and regularity results for the DiPerna-Lions flow.} J. Reine Angew. Math., $\mathbf{616}$, (2008), 15-46.

\bibitem{CNSS} \textsc{G. Crippa, C. Nobili, C. Seis, S. Spirito}: \emph{Eulerian and Lagrangian solutions to the continuity and Euler equations with $L^1$ vorticity}. SIAM J. Math. Anal., $\mathbf{49}$ (2017), no. 5, 3973-3998.

\bibitem{CS} \textsc{G. Crippa, S. Spirito}: \emph{Renormalized solutions of the 2D Euler equations.} Comm. Math. Phys., $\mathbf{339}$, (2015), 191-198.


\bibitem{De} \textsc{J.-M. Delort}: \emph{Existence de nappes de tourbillon en dimension deux.} J. Amer. Math. Soc., $\mathbf{4}$, (1991), 553-586.

\bibitem{DPL} \textsc{R. J. DiPerna, P.-L. Lions}: \emph{Ordinary differential equations, transport theory and Sobolev spaces.} Invententiones Mathematicae, $\mathbf{98}$, (1989), 511-547.

\bibitem{DPM} \textsc{R. J. DiPerna, A. Majda}: \emph{Concentrations in regularizations for 2-D incompressible flow.} Comm. Pure Appl. Math., $\mathbf{40}$, (1987), 301-345.


\bibitem{K} \textsc{H. Kunita}: \emph{Stochastic differential equations and stochastic flows of diffeomorphisms}. Lecture Notes in Math., Vol. 1097, Springer-Verlag (1984), 143-303.

\bibitem{L} \textsc{L. Lichtenstein}: \emph{\"Uber einige Existenzproblem der Hydrodynamik homogener, unzusammendr\"uckbarer, reibungsloser Fl\"ussigkeiten und die Helmholtzschen Wirbels\"atze.} M. Z., $\mathbf{23}$, (1925), 89-154.

\bibitem{LBL} \textsc{C. Le Bris, P.L. Lions}: \emph{Parabolic Equations with Irregular Data and Related Issues}.  De Gruyter Series in Applied and Numerical Mathematics 4, De Gruyter, (2019).  

\bibitem{FLT} \textsc{M. C. Lopes Filho, H. J. Nussenzveig Lopes, E. Tadmor}: \emph{Approximate solutions of the incompressible Euler equations with no concentrations.} Ann. Inst. H. Poincar\`e Anal. Non Lin\`eaire, $\mathbf{17}$, 3 (2000), 371-412.

\bibitem{FMN} \textsc{M. C. Lopes Filho, A. L. Mazzucato and H. J. Nussenzveig Lopes}: \emph{Weak solutions, renormalized solutions and enstrophy defects in 2D turbulence.} Arch. Ration. Mech. Anal., $\mathbf{179}$, no. 3 (2006), 353-387.

\bibitem{M} \textsc{A. J. Majda}: \emph{Remarks on Weak Solutions for Vortex Sheets with a Distinguished Sign.} Indiana Univ. Math. J., $\mathbf{42}$, (1993), 921-939.

\bibitem{MB} \textsc{A. J. Majda, A. L. Bertozzi}: \emph{Vorticity and incompressible flow}, vol. 27 of \emph{Cambridge Texts in Applied Mathematics}. Cambridge University Press, Cambridge, 2002.

\bibitem{M}\textsc{N. Masmoudi}: \emph{Remarks about the inviscid limit of the Navier-Stokes system}. Comm. Math. Phys. {\bf 270}, (2007), 777--788.

\bibitem{MeS} \textsc{F. Mengual, L. Sz\`ekelyhidi Jr}: \emph{Dissipative Euler flows for vortex sheet initial data without distinguished sign.} \url{https://arxiv.org/abs/2005.08333}

\bibitem{SW} \textsc{H. J. Nussenzveig Lopes, C. Seis, E. Wiedemann}: \emph{On the vanishing viscosity limit for 2D incompressible flows with unbounded vorticity.} \url{https://arxiv.org/abs/2007.01091}

\bibitem{Se} \textsc{C. Seis}: \emph{A note on the vanishing viscosity limit in the Yudovi\v{c} class.} Canadian Mathematical Bulletin, (2020), 1-11.

\bibitem{S} \textsc{P. Serfati}: \emph{Solutions $C^\infty$ en temps, $n$-log Lipschitz born\'ees en espace et \'equation d'Euler.} C. R. Acad. Sci. Paris S\'er. I Math., $\mathbf{320}$, (1995), 555-558.

\bibitem{Ste} \textsc{E. M. Stein}: \emph{Singular Integrals and Differentiability Properties of Functions}, Princeton University Press, (1970).

\bibitem{VW} \textsc{I. Vecchi, S. J. Wu}: \emph{On $L^1$-vorticity for 2-D incompressible flow.} Manuscripta Math. 78, $\mathbf{4}$, (1993), 403-412.

\bibitem{VI} \textsc{M. Vishik}: \emph{Instability and non-uniqueness in the Cauchy problem for the Euler equations of an ideal incompressible fluid. Part I}. \url{ https://arxiv.org/abs/1805.09426 }

\bibitem{VII} \textsc{M. Vishik}: \emph{Instability and non-uniqueness in the Cauchy problem for the Euler equations of an ideal incompressible fluid. Part II}. \url{ https://arxiv.org/abs/1805.09440 }

\bibitem{W} \textsc{W. Wolibner}: \emph{Un th\'eor\`eme sur l'existence du mouvement plan d'un fluide parfait, homogene, incompressible, pendant un temps infiniment long}. Math. Z., $\mathbf{37}$, (1933), 698-726.

\bibitem{Y} \textsc{V. I. Yudovi\v{c}}: \emph{Non-stationary flows of an ideal incompressible fluid.} \v{Z}. Vy\v{c}isl. Mat. i Mat. Fiz. 3, (1963), 1032-1066.
\end{thebibliography}
\end{document}